\documentclass[12pt]{amsart}
\usepackage{}

\usepackage{amsmath}
\usepackage{amsfonts}
\usepackage{amssymb}
\usepackage[all]{xy}           
\usepackage{bbm}
\usepackage{bbding}
\usepackage{txfonts}
\usepackage{amscd}
\usepackage{tikz}
\usetikzlibrary{matrix}
\usepackage[shortlabels]{enumitem}

\usepackage{tikz}
\usepackage[active]{srcltx}

\newcommand{\rmnum}[1]{\romannumeral #1}
\newcommand{\Rmnum}[1]{\expandafter\@slowromancap\romannumeral #1@}

\makeatletter

\newtheorem{thm}{Theorem}[section]
\newtheorem{lem}[thm]{Lemma}

\newtheorem{pro}[thm]{Proposition}
\newtheorem{ex}[thm]{Example}
\newtheorem{rmk}[thm]{Remark}
\newtheorem{defi}[thm]{Definition}

\newcommand {\emptycomment}[1]{}

\newcommand{\be }{\begin{equation}}
\newcommand{\ee }{\end{equation}}

\newcommand{\br}[1]{   [ \cdot,    \cdot  ]   }

\newcommand\blfootnote[1]{%
  \begingroup
  \renewcommand\thefootnote{}\footnote{#1}%
  \addtocounter{footnote}{-1}%
  \endgroup
}

\begin{document}

\setlength{\baselineskip}{1.2\baselineskip}
\title[Universal enveloping $H$-pseudoalgebras of DGP pseudoalgebras]
{Universal enveloping $H$-pseudoalgebras of DGP pseudoalgebras}

\author{Ying Chen}
\address{
School of Mathematical Sciences  \\
Zhejiang Normal University\\
Jinhua 321004\\
China}
\email{yingchen@zjnu.edu.cn}


\author{Jiafeng L\"u}
\address{
School of Mathematical Sciences    \\
Zhejiang Normal University\\
Jinhua 321004              \\
China}
\email{jiafenglv@zjnu.edu.cn}

\author{Jiaqun Wei}
\address{
School of Mathematical Sciences    \\
Zhejiang Normal University\\
Jinhua 321004              \\
China}
\email{weijiaqun5479@zjnu.edu.cn}
\blfootnote{*Corresponding Author: Jiaqun Wei. Email: weijiaqun5479@zjnu.edu.cn.}

\begin{abstract}
The notions of Poisson $H$-pseudoalgebras are generalizations of Poisson algebras in a pseudotensor category $\mathcal{M}^{\ast}(H)$. This paper introduces an analogue of Poisson-Ore extension in Poisson $H$-pseudoalgebras. Poisson $H$-pseudoalgebras with the differential graded setting induces the notions of differential graded Poisson $H$-pseudoalgebras (DGP pseudoalgebras, for short). The DGP pseudoalgebra with some compatibility conditions is proved to be closed under tensor product. Furthermore, the universal enveloping $H$-pseudoalgebras of DGP pseudoalgebras are constructed by a $\mathcal{P}$-triple. A unique differential graded pseudoalgebra homomorphism between a universal enveloping $H$-pseudoalgebra of a DGP pseudoalgebra and a $\mathcal{P}$-triple of a DGP pseudoalgebra is obtained.
\end{abstract}

\subjclass[2010]{ 
16E45, 
17B35, 
17B63, 
18D35. 
 }

\keywords{Poisson-Ore extension, Poisson $H$-pseudoalgebras, differential graded, universal enveloping $H$-pseudoalgebras.}

\maketitle
\tableofcontents
\allowdisplaybreaks

\section{Introduction and Main Statements}\label{sec:intr}

This paper aims to generalize Poisson-Ore extension in \cite{Oh-2006-3}, differential graded Poisson algebras and their universal enveloping algebras in \cite{Lv-2016-4} to the setting of Poisson $H$-pseudoalgebras \cite{Wang-2023-2}. Unless otherwise specified, $\mathbf{k}$ represents a field with characteristic $0$, $\mathbb{Z}$ denotes the set of integers.

\subsection{Poisson-Ore extension}

In 1933, Ore \cite{Ore-1933-0} introduced non-commutative polynomials rings, which are known as Ore extensions nowadays. Ore extensions have remarkable applications in various algebraic structures, including Hopf algebras \cite{Ore-Hopf-2003-17}, quantized Weyl algebras \cite{Ore-Weyl-2009-18}, quantum matrix algebras \cite{Ore-quantum-2018-19}.
Oh \cite{Oh-2006-3} first proposed the Poisson polynomial extensions as a Poisson version of Ore extensions. 

\begin{defi}
$($\cite[Section 1]{Oh-2006-3}$)$ Let $(A,\cdot,\{\cdot,\cdot\}_A)$ be a Poisson algebra and $\alpha:~A\rightarrow A$ is a linear map, for all $a,b\in A$, if $\alpha$ satisfies
\begin{itemize}
\item[(i)] $\alpha(a\cdot b)=\alpha(a)\cdot b+a\cdot\alpha(b);$
\item[(ii)] $\alpha(\{a,b\}_A)=\{\alpha(a),b\}_A+\{a,\alpha(b)\}_A,$
\end{itemize} 
then $\alpha$ is called a {\bf Poisson derivation} of $(A,\cdot,\{\cdot,\cdot\}_A)$. 
If $\alpha$ only satisfies (i), then $\alpha$ is called a {\bf derivation} of $(A,\cdot,\{\cdot,\cdot\}_A)$. 
\end{defi}

\begin{defi}$($\cite[Section 1]{Oh-2006-3}$)$
Let $(A,\cdot,\{\cdot,\cdot\}_A)$ be a Poisson algebra and $\alpha$ is a Poisson derivation of $(A,\cdot,\{\cdot,\cdot\}_A)$,
for all $a,b\in A$, if a linear map $\delta:~A\rightarrow A$ satisfies
\begin{itemize}
\item[(i)] $\delta(a\cdot b)=\delta(a)\cdot b+a\cdot\delta(b);$
\item[(ii)] $\delta(\{a,b\}_A)=\{\delta(a),b\}_A+\{a,\delta(b)\}_A+\delta(a)\cdot\alpha(b)-\alpha(a)\cdot\delta(b),$
\end{itemize} 
then $\delta$ is called a {\bf Poisson $\alpha$-derivation} of $(A,\cdot,\{\cdot,\cdot\}_A)$. 
\end{defi}

\begin{thm}\label{thm:Poisson-poly-ring}
$($\cite[Section 1]{Oh-2006-3}$)$ Let $(A,\cdot,\{\cdot,\cdot\}_A)$ be a Poisson algebra and $\alpha, \delta:~A\rightarrow A$ are two linear maps. Then the polynomial ring $A[x]$ is a Poisson algebra with Poisson bracket
\begin{equation}\label{eq:Poisson-poly-ring}
  \{a,b\}=\{a,b\}_A,\quad\quad \{a,x\}=\alpha(a)x+\delta(a),
\end{equation}
for all $a,b\in A$ if and only if $\alpha$ is a Poisson derivation of $(A,\cdot,\{\cdot,\cdot\}_A)$ and $\delta$ is a Poisson $\alpha$-derivation of $(A,\cdot,\{\cdot,\cdot\}_A)$. 
\end{thm}
The Poisson algebra $A[x]$ endowed with the Poisson bracket from Theorem \ref{thm:Poisson-poly-ring} is denoted by $A[x;\alpha,\delta]_P$ and called {\bf Poisson-Ore extension} of $(A,\cdot,\{\cdot,\cdot\}_A)$.
 
In \cite{Lv-2015-11, Lv-2018-12}, the authors introduced the notions of universal enveloping algebras of (double) Poisson-Ore extensions and proved that the universal enveloping algebra of a (double) Poisson-Ore extension is an iterated (double) Ore extension. 
In \cite{Goodearl-1982-9}, the authors computed the Krull dimension of Ore extension. Then Jordan \cite{Jordan-2014-10} generalized the Krull dimension to the Poisson-Ore extension.

\subsection{differential graded Poisson algebras}

The Poisson algebras originated from the study of Hamiltonian mechanics by Poisson. In both commutative and non-commutative algebraic contexts, the Poisson algebras have many important generalizations, such as noncommutative Poisson algebras \cite{Noncommutative-P-1994-14}, Novikov-Poisson algebras \cite{Novikov-P-1997-15} and quiver Poisson algebras \cite{Quiver-P-2007-16}.
In 2016, L\"{u}, Wang and Zhuang \cite{Lv-2016-4} defined differential graded Poisson algebras as a generalization of Poisson algebras in the differential graded setting.

\begin{defi}
Let $A=\bigoplus_{i\in \mathbb{Z}} A_i$ be a graded vector space. A {\bf graded algebra} is a pair $(A, \cdot)$, where $\cdot:~A\otimes A\rightarrow A$ is a $\mathbf{k}$-linear map satisfying $\mathbf{k}\subseteq A_0$ and $A_i\cdot A_j\subseteq A_{i+j}$, for all $i, j\in \mathbb{Z}$. Furthermore, if there is a $\mathbf{k}$-linear homogeneous map $d:~A\rightarrow A$ of degree 1 such
that $d^2=0$ and for all homogeneous elements $a, b\in A$,
\begin{equation}\label{gra-differ-product}
  d(a\cdot b)=d(a)\cdot b+(-1)^{|a|}a\cdot d(b),
\end{equation}
then $(A, \cdot, d)$ is called a {\bf differential graded algebra}.
\end{defi}
Moreover, for all homogeneous elements $a, b\in A$, if $a\cdot b=(-1)^{|a||b|}b\cdot a$, where $|a|$ and $|b|$ denote the degrees of $a$ and $b$ respectively, then graded algebra $(A, \cdot)$ and differential graded algebra $(A, \cdot, d)$ are called {\bf graded commutative}.

\begin{defi}
$($\cite[Definition 3.1]{Lv-2016-4}$)$ Let $(A, \cdot)$ be a graded commutative graded algebra. A {\bf graded Poisson algebra} is a triple $(A, \cdot, \{\cdot,\cdot\})$, where $\{\cdot, \cdot\}:~A\otimes A\rightarrow A$ is a $\mathbf{k}$-linear map of degree $p$, for all homogeneous elements $a, b, c \in A$,  such that 
\begin{enumerate}
\item[(\rmnum{1})] $(A, \{\cdot, \cdot\})$ is a graded Lie algebra, i.e., the bracket $\{\cdot, \cdot\}$ satisfies 
\begin{enumerate}
\item[] $\{a, b\}=-(-1)^{(|a|+p)(|b|+p)}\{b, a\}$,

\item[] $\{a, \{b, c\}\}=\{\{a, b\}, c\}+(-1)^{(|a|+p)(|b|+p)}\{b, \{a, c\}\}$;
\end{enumerate}


\item[(\rmnum{2})] (biderivation property): $\{a, b\cdot c\}=\{a, b\}\cdot c+(-1)^{(|a|+p)|b|}b\cdot \{a, c\}$.
\end{enumerate}
In addition, if there is a $\mathbf{k}$-linear homogeneous map $d:~A\rightarrow A$ of degree 1 such
that $(A, \cdot, d)$ is a differential graded algebra, $d^2=0$ and
\begin{enumerate}
\item[(\rmnum{3})] $d(\{a, b\})=\{d(a), b\}+(-1)^{(|a|+p)}\{a, d(b)\}$,

\end{enumerate}
then $(A, \cdot, \{\cdot,\cdot\}, d)$ is called a {\bf differential graded Poisson algebra}.
\end{defi}

In \cite{Lv-2016-4}, the authors constructed some examples of differential graded Poisson algebras from Lie theory, differential geometry, homological algebra and deformation theory. They also introduced the universal enveloping algebras of differential graded Poisson algebras.
In \cite{Hu-2019-13}, the authors showed that a differential graded symplectic ideal of a differential graded Poisson algebra is the annihilator of a simple differential graded Poisson module.

\subsection{Hopf algebras}\label{sec:Hopf}

In this paper, $H$ will always be a cocommutative Hopf algebra with a coproduct $\Delta$, a counit $\epsilon$ and a bijective antipode $S$. The inverse of $S$ is denoted by $S^{-1}$. 
We use Sweedler's notations: $\Delta(h)=\sum\limits_{h}h_{(1)}\otimes h_{(2)}$ and $(S\otimes1)\Delta(h)=\sum\limits_{h}h_{(-1)}\otimes h_{(2)}$ etc. The dual algebra of $H$ will always be denoted by $X$. 
For more details, we refer to \cite{Sweedler-1969-6}.




\subsection{Motivations}

Let $\mathcal{M}^{l}(H)$ be the category of all left $H$-modules, where $H$ is a cocommutative Hopf algebra over $\mathbf{k}$. 
Let $I$ be a finite nonempty set and denote the tensor product functor $\mathcal{M}^{l}(H)^{\otimes I}\rightarrow \mathcal{M}^{l}(H^{\otimes I})$ by $\boxtimes_{i\in I}$,
for any left $H$-modules $M$ and $L_i~(i\in I)$, then
\begin{equation}\label{eq:pseudo-catego}
  Lin(\{L_i\}_{i\in I}, M):=Hom_{H^{\otimes I}}(\boxtimes_{i\in I} L_i, H^{\otimes I}\otimes_H M)
\end{equation}
makes $\mathcal{M}^{l}(H)$ into a {\bf pseudotensor category} $\mathcal{M}^{\ast}(H)$ \cite{Bakalov-2001-1}. The pseudotensor category plays a crucial role in broad areas. For example, a pseudotensor category with an object is an operad, a Lie $H$-pseudoalgebra is a Lie algebra in the pseudotensor category, the primitive linearly compact Lie algebras in the pseudotensor category are called primitive pseudoalgebras of vector fields. For further details on pseudotensor category, see \cite{Bakalov-2001-1}. In \cite{Wang-2023-2}, the authors introduced the notions of Poisson $H$-pseudoalgebras, which can be viewed as Poisson algebras in the pseudotensor category, this Poisson $H$-pseudoalgebra is a Lie $H$-pseudoalgebra together with a commutative associative algebra, satisfying $H$-differential and Leibniz rule (Definition \ref{defi:Poi-pseudo}). 

A Poisson algebra can be regarded as a Lie algebra together with an associative algebra satisfying Leibniz rule. 
Thus, we consider to generalize Poisson derivations and Poisson-Ore extension of Poisson algebras to Poisson $H$-pseudoalgebras, introducing the notions of Poisson pseudoderivations and Poisson pseudo-Ore extension of Poisson $H$-pseudoalgebras respectively. 
Moreover, we generalize Poisson $H$-pseudoalgebras to the differential graded setting, obtaining differential graded Poisson $H$-pseudoalgebras and their universal enveloping $H$-pseudoalgebras.

\subsection{Outline of this paper}
The paper is organized as follows. Theorem \ref{thm:Poisson-pseudo-poly-ring} is the main result of Section \ref{sec:P-pseudo-Ore-exten}. 
In section \ref{sec:examples}, we first introduce some examples of Poisson $H$-pseudoalgebras on current $H$-pseudoalgebras. Then we construct a Poisson algebraic structure on an annihilation Lie algebra of a Lie $H$-pseudoalgebra in Proposition \ref{pro:annihi-Poi-alg}. 
In section \ref{sec:DGPpa-Uni}, we introduce the notions of differential graded Poisson $H$-pseudoalgebras (or DGP pseudoalgebras). We show that the tensor product of any two graded commutative DGP pseudoalgebras is closed with certain compatibility conditions (Proposition \ref{pro:tensor-psalg}).
Then we introduce the notions of universal enveloping $H$-pseudoalgebras of DGP pseudoalgebras.  It is proved that the universal enveloping $H$-pseudoalgebra of any DGP pseudoalgebra is a differential graded $H$-pseudoalgebra (Proposition \ref{pro:uni-enve-pseudoalg}). 
Furthermore, we introduce a $\mathcal{P}$-triple of a DGP pseudoalgebra, and prove that there is a unique differential graded pseudoalgebra homomorphism between a universal enveloping $H$-pseudoalgebra of a DGP pseudoalgebra and a $\mathcal{P}$-triple of a DGP pseudoalgebra (Theorem \ref{thm:bi-commute}).

\section{Poisson pseudo-Ore extension}\label{sec:P-pseudo-Ore-exten}

Theorem \ref{thm:Poisson-pseudo-poly-ring} is the main result of Section \ref{sec:P-pseudo-Ore-exten}. We generalize Theorem \ref{thm:Poisson-poly-ring} in \cite{Oh-2006-3} to the setting of Poisson $H$-pseudoalgebras, and define a Poisson pseudo-Ore extension of a Poisson $H$-pseudoalgebra.

\begin{defi}
$($\cite[Section 1]{Bakalov-2001-1}$)$ An {\bf $H$-pseudoalgebra} is a left $H$-module $A$ endowed with a left $H$-bilinear map $\mu:~A\otimes A\rightarrow (H\otimes H)\otimes_H A$ satisfying:
\begin{equation}\label{eq:psalg-bili}
  fa\ast gb=\big((f\otimes g)\otimes_H \mathrm{id}_A\big)(a\ast b),\quad\forall a, b\in A,~f, g\in H,
\end{equation}
where the identity map of $A$ is denoted by $\mathrm{id}_A$, this left $H$-bilinear map $\mu$ is called a {\bf pseudoproduct} and denoted $\mu(a\otimes b)$ by $a\ast b$.
\end{defi}
Moreover, if $(a\ast b)\ast c=a\ast(b\ast c)$, for all $a, b, c\in A$, then this $H$-pseudoalgebra is {\bf associative}, we denote it by $(A,\ast)$.

\begin{lem}\label{lem:basis}
$($\cite[Section 2]{Bakalov-2001-1}$)$ Every element of $H\otimes H$ can be uniquely represented as $\sum\limits_i(h_i\otimes 1)\Delta(l_i)$, where $\{h_i\}$ is a fixed $\mathbf{k}$-basis of $H$ and $l_i\in H$.
\end{lem}
Hence the pseudoproduct on $(A,\ast)$ can be written in the form:
\begin{equation*}
  a\ast b=\sum\limits_i(h_i\otimes 1)\otimes_H c_i.
\end{equation*}

\begin{defi}
$($\cite[Section 3]{Bakalov-2001-1}$)$ An {\bf Lie $H$-pseudoalgebra} is a left $H$-module $A$ endowed with a left $H$-bilinear map $\mu:~A\otimes A\rightarrow (H\otimes H)\otimes_H A$, satisfying the following axioms, for all $a, b, c \in A,~f, g \in H,~\sigma=(12) \in S_2$, 
\begin{align}
  \label{eq:H-bilinear}\bf \textit{H}-bilinearity&: \{fa\ast gb\}=\big((f\otimes g)\otimes_H\mathrm{id}_A\big)\{a\ast b\}, \\
  \label{eq:Skewsymmetry}\bf Skew-symmetry&: \{b\ast a\}=-(\sigma\otimes_H \mathrm{id}_A)\{a\ast b\},\\
  \label{eq:Jacobi-id}\bf Jacobi~~identity&: \{a\ast\{b\ast c\}\}-\big((\sigma\otimes \mathrm{id})\otimes_H \mathrm{id}_A\big)\{b\ast\{a\ast c\}\}=\{\{a\ast b\}\ast c\},
\end{align}
where the identity map of $H$ is denoted by $\mathrm{id}$, the pseudoproduct $\mu$ is called a {\bf pseudobracket} and denoted by $\mu(a\otimes b)=\{a\ast b\}$. We denote this Lie $H$-pseudoalgebra by $(A,\{\ast \})$.
\end{defi}

\begin{defi}\label{defi:Poi-pseudo}
$($\cite[Section 4]{Wang-2023-2}$)$ A {\bf Poisson $H$-pseudoalgebra} $(A,\cdot,\{\ast\})$ is a Lie $H$-pseudoalgebra $(A,\{\ast\})$ equipped with an associative product $\cdot:~A\otimes A\rightarrow A$, satisfying the following axioms, for all $a, b, c \in A,~h \in H$,
\begin{align}
  \label{eq:H-differential}\bf \textit{H}-differential&: h(a\cdot b)=(h_{(1)}a)\cdot(h_{(2)}b), \\
  \label{eq:left-Leibniz-rule}\bf left~~Leibniz~~rule&: \{a\ast (b\cdot c)\}=\{a\ast b\}\cdot c+\{a\ast c\}\cdot b,
\end{align}
for $\{a\ast b\}=\sum\limits_{i}(f_i\otimes g_i)\otimes_H e_i$,~$\{a\ast b\}\cdot c$ is defined by
\begin{equation}\label{eq:left-Leibniz}
  \{a\ast b\}\cdot c:= \sum\limits_{i}(f_i\otimes g_{i(1)})\otimes_H e_i\cdot(g_{i(-2)}c).
\end{equation}
\end{defi}
By \eqref{eq:Skewsymmetry} and \eqref{eq:left-Leibniz}, we have
\begin{equation*}
  \{b\ast a\}\cdot c=-\sum\limits_{i}(g_i\otimes f_{i(1)})\otimes_H e_i\cdot(f_{i(-2)}c).
\end{equation*}
Furthermore, if $(A,\cdot)$ is a commutative associative algebra, then Poisson $H$-pseudoalgebra $(A,\cdot,\{\ast\})$ is {\bf commutative}.

\begin{lem}\label{lem:right-Leibniz-rule}
$($\cite[Section 4]{Wang-2023-2}$)$ Let $(A,\cdot,\{\ast\})$ be a Poisson $H$-pseudoalgebra, then we have the {\bf right Leibniz rule}
\begin{equation}\label{eq:right-Leibniz-rule}
  \{(a\cdot b)\ast c\}=a\cdot\{b\ast c\}+b\cdot\{a\ast c\},\quad\forall a,b,c \in A.
\end{equation}

\end{lem}
Similar to \eqref{eq:left-Leibniz}, for $\{b\ast c\}=\sum\limits_{i}(h_i\otimes l_i)\otimes_H d_i$,~$a\cdot\{b\ast c\}$ is defined by
\begin{equation}\label{eq:right-Leibniz}
  a\cdot\{b\ast c\}:=\sum\limits_{i}(h_{i(1)}\otimes l_i)\otimes_H (h_{i(-2)}a)\cdot d_i.
\end{equation}
By \eqref{eq:Skewsymmetry} and \eqref{eq:right-Leibniz}, we have
\begin{equation*}
  a\cdot\{c\ast b\}:=-\sum\limits_{i}(l_{i(1)}\otimes h_i)\otimes_H (l_{i(-2)}a)\cdot d_i.
\end{equation*}

Unless otherwise specified, for any left $H$-module $A$, we omit the summation and simplify $\sum\limits_{i}(f_i\otimes g_i)\otimes_H a_i\in H^{\otimes2}\otimes_H A$ as $(f_i\otimes g_i)\otimes_H a_i$ in the following page.
 
In the definition of Poisson $H$-pseudoalgebra $(A,\cdot,\{\ast\})$, for all $a, b\in A$ and $f_1, f_2,...,f_{n-1}, f_n\in H$, we define
\begin{eqnarray}
  \label{eqs:define-products1}(f_1\otimes f_2\otimes\cdots\otimes f_n\otimes_H a)\cdot b &:=& (f_1\otimes f_2f_{n(1)}\otimes\cdots\otimes f_{n-1}f_{n({n-2})}\otimes f_{n(n-1)})\otimes_H a\cdot (f_{n(-n)}b), \\
  \label{eqs:define-products2}a\cdot(f_1\otimes f_2\otimes\cdots\otimes f_n\otimes_H b) &:=& (f_{1(1)}\otimes f_2f_{1(2)}\otimes\cdots\otimes f_{n-1}f_{1(n-1)}\otimes f_n)\otimes_H (f_{1(-n)}a)\cdot b,
\end{eqnarray}
where $\Delta^{n-1}(f_1)=f_{1(1)}\otimes f_{1(2)}\cdots f_{1(n-1)}\otimes f_{1(n)}$, and we use the following rule:
\begin{equation}\label{eq:Delta-rule}
  \{(f\otimes_H a)\ast(g\otimes_H b)\}:=(f\Delta^{n-1}\otimes g\Delta^{m-1}\otimes_H \mathrm{id}_A)\{a\ast b\},
\end{equation}
for $f\otimes_H a \in H^{\otimes n}\otimes_H A$ and $g\otimes_H b \in H^{\otimes m}\otimes_H A$. Throughout this paper, we use rule \eqref{eq:Delta-rule} to deal with pseudobracket $\{\ast\}$.

\begin{defi}\label{defi:Poi-pseudo-deri}
Let $(A,\cdot,\{\ast\}_A)$ be a Poisson $H$-pseudoalgebra and $\alpha:~A\rightarrow A$ is a left $H$-linear map, for all $a,b\in A$, if $\alpha$ satisfies
\begin{itemize}
\item[(i)] $\alpha(a\cdot b)=\alpha(a)\cdot b+a\cdot\alpha(b);$
\item[(ii)] $(\mathrm{id}\otimes \mathrm{id}\otimes_H \alpha)\{a\ast b\}_A=\{\alpha(a)\ast b\}_A+\{a\ast\alpha(b)\}_A,$
\end{itemize} 
then $\alpha$ is called a {\bf Poisson pseudoderivation} of $(A,\cdot,\{\ast\}_A)$. 
If $\alpha$ only satisfies (i), then $\alpha$ is called a {\bf pseudoderivation} of $(A,\cdot,\{\ast\}_A)$. 
\end{defi}

\begin{defi}\label{defi:Poi-alpha-pseudo-deri}
Let $(A,\cdot,\{\ast\}_A)$ be a Poisson $H$-pseudoalgebra and $\alpha$ is a Poisson pseudoderivation of $(A,\cdot,\{\ast\}_A)$. 
Let $\delta:~A\rightarrow A$ be a left $H$-linear map, for all $a,b\in A$, if $\delta$ satisfies
\begin{itemize}
\item[(i)] $\delta(a\cdot b)=\delta(a)\cdot b+a\cdot\delta(b);$
\item[(ii)] $(\mathrm{id}\otimes \mathrm{id}\otimes_H \delta)\{a\ast b\}_A=\{\delta(a)\ast b\}_A+\{a\ast\delta(b)\}_{A}
+(1\otimes 1)\otimes_H\delta(a)\cdot\alpha(b)-(1\otimes 1)\otimes_H\alpha(a)\cdot\delta(b),$
\end{itemize} 
where the unit of $H$ is denoted by $1$, then $\delta$ is called a {\bf Poisson $\alpha$-pseudoderivation} of $(A,\cdot,\{\ast\}_A)$. 
\end{defi}

\begin{thm}\label{thm:Poisson-pseudo-poly-ring}
Let $(A,\cdot,\{\ast\}_A)$ be a commutative Poisson $H$-pseudoalgebra, then the polynomial ring $A[x]$ is a Poisson $H$-pseudoalgebra with Poisson pseudobracket
\begin{equation}\label{eq:Poisson-pseudo-poly-ring}
  \{a\ast b\}=\{a\ast b\}_A,\quad\quad \{a\ast x\}=(1\otimes 1)\otimes_H\big(\alpha(a)x+\delta(a)\big),
\end{equation}
for all $a,b\in A$ if and only if $\alpha$ is a Poisson pseudoderivation of $(A,\cdot,\{\ast\}_A)$ and $\delta$ is a Poisson $\alpha$-pseudoderivation of $(A,\cdot,\{\ast\}_A)$.
\end{thm}

\begin{proof}
If $A[x]$ is a Poisson $H$-pseudoalgebra with the Poisson pseudobracket \eqref{eq:Poisson-pseudo-poly-ring}, for all $a,b\in A$, then we obtain 
\begin{eqnarray*}
  \{(a\cdot b)\ast x\} &=& (1\otimes 1)\otimes_H\big(\alpha(a\cdot b)x+\delta(a\cdot b)\big), \\
  a\{b\ast x\}+b\{a\ast x\} &=& (1\otimes 1)\otimes_H\big(a(\alpha(b)x)+a\cdot\delta(b)\big)+(1\otimes 1)\otimes_H\big(b(\alpha(a)x)+b\cdot\delta(a)\big)\\
  &=&(1\otimes 1)\otimes_H\big(a\cdot\alpha(b)+b\cdot\alpha(a)\big)x+(1\otimes 1)\otimes_H\big(a\cdot\delta(b)+b\cdot\delta(a)\big),
\end{eqnarray*}
by \eqref{eq:right-Leibniz-rule} and the commutativity of associative algebra $(A,\cdot)$, thus both $\alpha$ and $\delta$ are pseudoderivations of $(A,\cdot,\{\ast\}_A)$. 

By Lemma \ref{lem:basis}, suppose that $\{a\ast b\}_A=(h_i\otimes1)\otimes_H c_i$, $\{a\ast\alpha(b)\}_A=(h_j\otimes1)\otimes_H c_j$, $\{a\ast\delta(b)\}_A=(h_k\otimes1)\otimes_H c_k$, $\{b\ast\alpha(a)\}_A=(l_p\otimes1)\otimes_H e_p$, $\{b\ast\delta(a)\}_A=(l_q\otimes1)\otimes_H e_q$. By \eqref{eq:left-Leibniz-rule}, \eqref{eqs:define-products1}, \eqref{eq:Delta-rule} and \eqref{eq:Poisson-pseudo-poly-ring}, then we have
\begin{eqnarray*}
  &&\{\{a\ast b\}_A\ast x\} = \{\big((h_i\otimes1)\otimes_H c_i\big) \ast x\}\\
  &=&(h_i\otimes1\otimes1)(\Delta\otimes1)(1\otimes1)\otimes_H \alpha(c_i)x+(h_i\otimes1\otimes1)(\Delta\otimes1)(1\otimes1)\otimes_H \delta(c_i)\\
  &=&\big((h_i\otimes1\otimes1)\otimes_H\alpha(c_i)\big)x+(h_i\otimes1\otimes1)\otimes_H\delta(c_i),
\end{eqnarray*}
\begin{eqnarray*}
  &&\{a\ast\{b\ast x\}\} = \{a\ast \big((1\otimes1)\otimes_H\alpha(b)x\big)\}+\{a\ast \big((1\otimes 1)\otimes_H\delta(b)\big)\}  \\
  &=&(1\otimes1\otimes1)(1\otimes\Delta)(h_j\otimes1)\otimes_H c_jx
  +(1\otimes1\otimes1)(1\otimes\Delta)(1\otimes1)\otimes_H(\alpha(a)x)\alpha(b)\\
  &&+(1\otimes1\otimes1)(1\otimes\Delta)(1\otimes1)\otimes_H\delta(a)\cdot\alpha(b)
  +(1\otimes1\otimes1)(1\otimes\Delta)(h_k\otimes1)\otimes_H c_k\\
  &=&\big((h_j\otimes1\otimes1)\otimes_Hc_j\big)x+(1\otimes1\otimes1)\otimes_H(\alpha(a)x)\alpha(b)\\
  &&+(1\otimes1\otimes1)\otimes_H\delta(a)\cdot\alpha(b)+(h_k\otimes1\otimes1)\otimes_Hc_k,
\end{eqnarray*}
\noindent and
\begin{eqnarray*}
  &&\{b\ast\{a\ast x\}\} = \{b\ast ((1\otimes1)\otimes_H\alpha(a)x)\}+\{b\ast ((1\otimes 1)\otimes_H\delta(a))\}  \\
  &=&(1\otimes1\otimes1)(1\otimes\Delta)(l_p\otimes1)\otimes_H e_px
  +(1\otimes1\otimes1)(1\otimes\Delta)(1\otimes1)\otimes_H(\alpha(b)x)\alpha(a)\\
  &&+(1\otimes1\otimes1)(1\otimes\Delta)(1\otimes1)\otimes_H\delta(b)\cdot\alpha(a)
  +(1\otimes1\otimes1)(1\otimes\Delta)(l_q\otimes1)\otimes_H e_q\\
  &=&\big((l_p\otimes1\otimes1)\otimes_He_p\big)x+(1\otimes1\otimes1)\otimes_H(\alpha(b)x)\alpha(a)\\
  &&+(1\otimes1\otimes1)\otimes_H\delta(b)\cdot\alpha(a)+(l_q\otimes1\otimes1)\otimes_He_q.
\end{eqnarray*}
Since the Poisson pseudobracket $\{\ast\}$ satisfies the Jacobi identity, thus we obtain 
\begin{eqnarray*}
  0 &=& \{\{a\ast b\}_A\ast x\}-\{a\ast\{b\ast x\}\}+\big((\sigma\otimes\mathrm{id})\otimes_H \mathrm{id}\big)\{b\ast\{a\ast x\}\} \\
  &=& \big((h_i\otimes1\otimes1)\otimes_H\alpha(c_i)\big)x+(h_i\otimes1\otimes1)\otimes_H\delta(c_i)\\
  &&-\big((h_j\otimes1\otimes1)\otimes_Hc_j\big)x-(1\otimes1\otimes1)\otimes_H(\alpha(a)x)\alpha(b)\\
  &&-(1\otimes1\otimes1)\otimes_H\delta(a)\cdot\alpha(b)-(h_k\otimes1\otimes1)\otimes_Hc_k\\
  &&+\big((1\otimes l_p\otimes1)\otimes_He_p\big)x+(1\otimes1\otimes1)\otimes_H(\alpha(b)x)\alpha(a)\\
  &&+(1\otimes1\otimes1)\otimes_H\delta(b)\cdot\alpha(a)+(1\otimes l_q\otimes1)\otimes_He_q\\
  &=&\big((h_i\otimes1\otimes1)\otimes_H\alpha(c_i)-(h_j\otimes1\otimes1)\otimes_Hc_j+(1\otimes l_p\otimes1)\otimes_He_p\big)x\\
  &&+(h_i\otimes1\otimes1)\otimes_H\delta(c_i)-(h_k\otimes1\otimes1)\otimes_Hc_k+(1\otimes l_q\otimes1)\otimes_He_q\\
  &&+(1\otimes1\otimes1)\otimes_H\delta(b)\cdot\alpha(a)-(1\otimes1\otimes1)\otimes_H\delta(a)\cdot\alpha(b)\\
  &&+(1\otimes1\otimes1)\otimes_H(\alpha(b)x)\alpha(a)-(1\otimes1\otimes1)\otimes_H(\alpha(a)x)\alpha(b).
\end{eqnarray*}
By the commutativity of associative algebra $(A,\cdot)$, we have 
\begin{equation*}
  (\alpha(b)x)\alpha(a)=(\alpha(b)\cdot\alpha(a))x=(\alpha(a)\cdot\alpha(b))x=(\alpha(a)x)\alpha(b),
\end{equation*}
and
\begin{eqnarray*}
  0&=&(h_i\otimes1)\otimes_H\alpha(c_i)-(h_j\otimes1)\otimes_Hc_j+(1\otimes l_p)\otimes_He_p\\
   &=&(\mathrm{id}\otimes \mathrm{id}\otimes_H \alpha)\{a\ast b\}_A-\{a\ast\alpha(b)\}_A-\{\alpha(a)\ast b\}_A,\\
  0&=&(h_i\otimes1)\otimes_H\delta(c_i)-(h_k\otimes1)\otimes_Hc_k+(1\otimes l_q)\otimes_He_q\\
  &&+(1\otimes1)\otimes_H\delta(b)\cdot\alpha(a)-(1\otimes1)\otimes_H\delta(a)\cdot\alpha(b)\\
   &=&(\mathrm{id}\otimes \mathrm{id}\otimes_H\delta)\{a\ast b\}_A-\{a\ast\delta(b)\}_{A}-\{\delta(a)\ast b\}_A\\
  &&-(1\otimes 1)\otimes_H\delta(a)\cdot\alpha(b)+(1\otimes 1)\otimes_H\alpha(a)\cdot\delta(b),
\end{eqnarray*}
thus $\alpha$ is a Poisson pseudoderivation of $(A,\cdot,\{\ast\}_A)$ and $\delta$ is a Poisson $\alpha$-pseudoderivation of $(A,\cdot,\{\ast\}_A)$.
  
Conversely, suppose that $\alpha$ is a Poisson pseudoderivation of $(A,\cdot,\{\ast\}_A)$ and $\delta$ is a Poisson $\alpha$-pseudoderivation of $(A,\cdot,\{\ast\}_A)$. 
Define a left $H$-bilinear map $\{\ast\}:~A[x]\otimes A[x]\rightarrow (H\otimes H)\otimes_HA[x]$ by
\begin{eqnarray}\begin{split}\label{eq:pseu-duct-poly-ring}
\{ax^i \ast bx^j\} =& \big(\{a\ast b\}_A+(1\otimes1)\otimes_Hjb\cdot\alpha(a)-(1\otimes1)\otimes_Hia\cdot\alpha(b)\big)x^{i+j} \\
                    &+ \big((1\otimes1)\otimes_Hjb\cdot\delta(a)-(1\otimes1)\otimes_Hia\cdot\delta(b)\big)x^{i+j-1},
                \end{split}
\end{eqnarray}
for all monomials $ax^i,~bx^j\in A[x]$.
Note that for $i=0,~j=1$ and $b=1$ in \eqref{eq:pseu-duct-poly-ring}, the case is \eqref{eq:Poisson-pseudo-poly-ring}. 
By \eqref{eq:pseu-duct-poly-ring}, we have $\{ax^i \ast bx^j\}=-(\sigma\otimes_H\mathrm{id}_A)\{bx^j \ast ax^i\}$. 

For simplicity, we will consistently denote an identity map by $\mathrm{id}$ without confusion, for all $ax^i,~bx^j,~cx^k\in A[x]$, it suffices to check the Jacobi identity:
\begin{equation}\label{eq:Jacobi-iden-A[x]}
  \{\{ax^i\ast bx^j\}\ast cx^{k}\}-\{ax^i\ast\{bx^j \ast cx^{k}\}\}+\big((\sigma\otimes\mathrm{id})\otimes_H \mathrm{id}\big)\{bx^j\ast\{ax^i \ast cx^{k}\}\}=0.
\end{equation}
We prove the above Jacobi identity by the induction on $i, j, k$. The trivial case $i=j=k=0$ in \eqref{eq:Jacobi-iden-A[x]} equals to \eqref{eq:Jacobi-id}.
The case $k=c=1$ and $j=i=0$ in \eqref{eq:Jacobi-iden-A[x]} equals to
\begin{equation*}
  \{\{a\ast b\}_A\ast x\}-\{a\ast\{b \ast x\}\}+\big((\sigma\otimes\mathrm{id})\otimes_H \mathrm{id}\big)\{b\ast\{a\ast x\}\}=0.
\end{equation*}
It is shown immediately by the commutativity of associative algebra $(A,\cdot)$, $\alpha$ satisfies (ii) in Definition \ref{defi:Poi-pseudo-deri} and $\delta$ satisfies (ii) in Definition \ref{defi:Poi-alpha-pseudo-deri}.
Suppose that \eqref{eq:Jacobi-iden-A[x]} holds for any $i, j, k$, then we show that \eqref{eq:Jacobi-iden-A[x]} also holds for $i, j, k+1$ by induction:
\begin{eqnarray*}
  \{\{ax^i\ast bx^j\}\ast cx^{k+1}\} &=& \{\{a\ast b\}_Ax^{i+j} \ast cx^{k+1}\}\\
  &&+\{((1\otimes1)\otimes_Hjb\cdot\alpha(a))x^{i+j} \ast cx^{k+1}\}-\{((1\otimes1)\otimes_Hia\cdot\alpha(b))x^{i+j}\ast cx^{k+1}\}\\
  &&+\{((1\otimes1)\otimes_Hjb\cdot\delta(a))x^{i+j-1}\ast cx^{k+1}\}-\{((1\otimes1)\otimes_Hia\cdot\delta(b))x^{i+j-1}\ast cx^{k+1}\}\\
  &=&\{\{a\ast b\}_Ax^{i+j} \ast cx^{k}\}x+\{\{a\ast b\}_Ax^{i+j} \ast x\}cx^{k}\\
  &&+\{(1\otimes1)\otimes_Hjb\cdot\alpha(a)x^{i+j} \ast cx^{k}\}x
  +1^{\otimes3}\otimes_{H}c\cdot\alpha\big(jb\cdot\alpha(a)\big)x^{i+j+k+1}\\
  &&+1^{\otimes3}\otimes_{H}c\cdot\delta\big(jb\cdot\alpha(a)\big)x^{i+j+k}
  -\{(1\otimes1)\otimes_Hia\cdot\alpha(b)x^{i+j}\ast cx^{k}\}x\\
  &&-1^{\otimes3}\otimes_{H}c\cdot\alpha\big(ia\cdot\alpha(b)\big)x^{i+j+k+1}
  -1^{\otimes3}\otimes_{H}c\cdot\delta\big(ia\cdot\alpha(b)\big)x^{i+j+k}\\
  &&+\{(1\otimes1)\otimes_Hjb\cdot\delta(a)x^{i+j-1} \ast cx^{k}\}x
  +1^{\otimes3}\otimes_{H}c\cdot\alpha\big(jb\cdot\delta(a)\big)x^{i+j+k}\\
  &&+1^{\otimes3}\otimes_{H}c\cdot\delta\big(jb\cdot\delta(a)\big)x^{i+j+k-1}
  -\{(1\otimes1)\otimes_Hia\cdot\delta(b)x^{i+j-1}\ast cx^{k}\}x\\
  &&-1^{\otimes3}\otimes_{H}c\cdot\alpha\big(ia\cdot\delta(b)\big)x^{i+j+k}
  -1^{\otimes3}\otimes_{H}c\cdot\delta\big(ia\cdot\delta(b)\big)x^{i+j+k-1}\\
  &=&\{\{ax^i\ast bx^j\}\ast cx^{k}\}x+\{\{a\ast b\}_Ax^{i+j} \ast x\}cx^{k}\\
  &&+1^{\otimes3}\otimes_{H}c\cdot\alpha\big(jb\cdot\alpha(a)\big)x^{i+j+k+1}
  +1^{\otimes3}\otimes_{H}c\cdot\delta\big(jb\cdot\alpha(a)\big)x^{i+j+k}\\
  &&-1^{\otimes3}\otimes_{H}c\cdot\alpha\big(ia\cdot\alpha(b)\big)x^{i+j+k+1}
  -1^{\otimes3}\otimes_{H}c\cdot\delta\big(ia\cdot\alpha(b)\big)x^{i+j+k}\\
  &&+1^{\otimes3}\otimes_{H}c\cdot\alpha\big(jb\cdot\delta(a)\big)x^{i+j+k}
  +1^{\otimes3}\otimes_{H}c\cdot\delta\big(jb\cdot\delta(a)\big)x^{i+j+k-1}\\
  &&-1^{\otimes3}\otimes_{H}c\cdot\alpha\big(ia\cdot\delta(b)\big)x^{i+j+k}
  -1^{\otimes3}\otimes_{H}c\cdot\delta\big(ia\cdot\delta(b)\big)x^{i+j+k-1}\\
  &=&\{\{ax^i\ast bx^j\}\ast cx^{k}\}x+\{\{a\ast b\}_Ax^{i+j} \ast x\}cx^{k}\\
  &&+1^{\otimes3}\otimes_{H}\alpha\big(jb\cdot\alpha(a)\big)x^{i+j+1}cx^{k}
  +1^{\otimes3}\otimes_{H}\delta\big(jb\cdot\alpha(a)\big)x^{i+j}cx^{k}\\
  &&-1^{\otimes3}\otimes_{H}\alpha\big(ia\cdot\alpha(b)\big)x^{i+j+1}cx^{k}
  -1^{\otimes3}\otimes_{H}\delta\big(ia\cdot\alpha(b)\big)x^{i+j}cx^{k}\\
  &&+1^{\otimes3}\otimes_{H}\alpha\big(jb\cdot\delta(a)\big)x^{i+j}cx^{k}
  +1^{\otimes3}\otimes_{H}\delta\big(jb\cdot\delta(a)\big)x^{i+j-1}cx^{k}\\
  &&-1^{\otimes3}\otimes_{H}\alpha\big(ia\cdot\delta(b)\big)x^{i+j}cx^{k}
  -1^{\otimes3}\otimes_{H}\delta\big(ia\cdot\delta(b)\big)x^{i+j-1}cx^{k}\\
  &=&\{\{ax^i\ast bx^j\}\ast cx^{k}\}x+\{\{ax^{i}\ast bx^{j}\} \ast x\}cx^{k}.
\end{eqnarray*}
Since
\begin{eqnarray*}
  \{bx^j \ast cx^{k+1}\} &=& \{bx^j \ast cx^{k}\}x+(1\otimes1)\otimes_H c\cdot\alpha(b)x^{j+k+1}
  +(1\otimes1)\otimes_H c\cdot\delta(b)x^{j+k} \\
  &=& \{bx^j \ast cx^{k}\}x+\big((1\otimes1)\otimes_H \alpha(b)x^{j+1}\big)cx^{k}
  +\big((1\otimes1)\otimes_H\delta(b)x^{j}\big)cx^{k} \\
  &=& \{bx^j \ast cx^{k}\}x+\{bx^j \ast x\}cx^{k},
\end{eqnarray*}
suppose $\{b\ast c\}_A=(f_t\otimes1)\otimes_H m_t$, $\{c\ast a\}_A=(g_s\otimes1)\otimes_H n_s$, then we have
\begin{eqnarray*}
  \{\{bx^j \ast cx^{k+1}\}\ast ax^i\} &=& \{(\{bx^j \ast cx^{k}\}x)\ast ax^i\}+\{(\{bx^j \ast x\}cx^{k})\ast ax^i\}  \\
  &=& \{\{bx^j \ast cx^{k}\}\ast ax^i\}x\\
  &&-(f_t\otimes1\otimes1)\otimes_H m_t\cdot\alpha(a)x^{j+k+i+1}-(f_t\otimes1\otimes1)\otimes_H m_t\cdot\delta(a)x^{j+k+i}\\
  &&-1^{\otimes3}\otimes_{H}kc\cdot\alpha(b)\cdot\alpha(a)x^{j+k+i+1}-1^{\otimes3}\otimes_{H}kc\cdot\alpha(b)\cdot\delta(a)x^{j+k+i}\\
  &&+1^{\otimes3}\otimes_{H}jb\cdot\alpha(c)\cdot\alpha(a)x^{j+k+i+1}+1^{\otimes3}\otimes_{H}jb\cdot\alpha(c)\cdot\delta(a)x^{j+k+i}\\
  &&-1^{\otimes3}\otimes_{H}kc\cdot\delta(b)\cdot\alpha(a)x^{j+k+i}-1^{\otimes3}\otimes_{H}kc\cdot\delta(b)\cdot\delta(a)x^{j+k+i-1}\\
  &&+1^{\otimes3}\otimes_{H}jb\cdot\delta(c)\cdot\alpha(a)x^{j+k+i}+1^{\otimes3}\otimes_{H}jb\cdot\delta(c)\cdot\delta(a)x^{j+k+i-1}\\
  &&+\{\{bx^j \ast x\}\ast ax^i\}cx^{k}\\
  &&+\big(g_{s(1)(1)}\otimes g_{s(1)(2)}\otimes1\big)\otimes_H(g_{s(-2)}\alpha(b))\cdot n_sx^{j+k+i+1}\\
  &&+1^{\otimes3}\otimes_{H}ia\cdot\alpha(b)\cdot\alpha(c)x^{j+k+i+1}+1^{\otimes3}\otimes_{H}ia\cdot\alpha(b)\cdot\delta(c)x^{j+k+i}\\
  &&-1^{\otimes3}\otimes_{H}k(\alpha(b)\cdot c)\cdot\alpha(a)x^{j+k+i+1}
  -1^{\otimes3}\otimes_{H}k(\alpha(b)\cdot c)\cdot\delta(a)x^{j+k+i}\\
  &&+\big(g_{s(1)(1)}\otimes g_{s(1)(2)}\otimes1\big)\otimes_H(g_{s(-2)}\delta(b))\cdot n_sx^{j+k+i}\\
  &&+1^{\otimes3}\otimes_{H}ia\cdot\delta(b)\cdot\alpha(c)x^{j+k+i}+1^{\otimes3}\otimes_{H}ia\cdot\delta(b)\cdot\delta(c)x^{j+k+i-1}\\
  &&-1^{\otimes3}\otimes_{H}k(\delta(b)\cdot c)\cdot\alpha(a)x^{j+k+i}
  -1^{\otimes3}\otimes_{H}k(\delta(b)\cdot c)\cdot\delta(a)x^{j+k+i-1},
\end{eqnarray*}
and
\begin{eqnarray*}
  \{\{ax^i \ast cx^{k+1}\}\ast bx^j\} &=& \{(\{ax^i \ast cx^{k}\}x)\ast bx^j\}+\{(\{ax^i \ast x\}cx^{k})\ast bx^j\}  \\
  &=& \{\{ax^i \ast cx^{k}\}\ast bx^j\}x\\
  &&+(1\otimes g_s\otimes1)\otimes_H n_s\cdot\alpha(b)x^{j+k+i+1}+(1\otimes g_s\otimes1)\otimes_H n_s\cdot\delta(b)x^{j+k+i}\\
  &&-1^{\otimes3}\otimes_{H}kc\cdot\alpha(a)\cdot\alpha(b)x^{j+k+i+1}-1^{\otimes3}\otimes_{H}kc\cdot\alpha(a)\cdot\delta(b)x^{j+k+i}\\
  &&+1^{\otimes3}\otimes_{H}ia\cdot\alpha(c)\cdot\alpha(b)x^{j+k+i+1}+1^{\otimes3}\otimes_{H}ia\cdot\alpha(c)\cdot\delta(b)x^{j+k+i}\\
  &&-1^{\otimes3}\otimes_{H}kc\cdot\delta(a)\cdot\alpha(b)x^{j+k+i}-1^{\otimes3}\otimes_{H}kc\cdot\delta(a)\cdot\delta(b)x^{j+k+i-1}\\
  &&+1^{\otimes3}\otimes_{H}ia\cdot\delta(c)\cdot\alpha(b)x^{j+k+i}+1^{\otimes3}\otimes_{H}ia\cdot\delta(c)\cdot\delta(b)x^{j+k+i-1}\\
  &&+\{\{ax^i \ast x\}\ast bx^j\}cx^{k}\\
  &&-(1\otimes 1\otimes f_t)\otimes_H\alpha(a)\cdot m_tx^{j+k+i+1}\\
  &&+1^{\otimes3}\otimes_{H}jb\cdot\alpha(a)\cdot\alpha(c)x^{j+k+i+1}+1^{\otimes3}\otimes_{H}jb\cdot\alpha(a)\cdot\delta(c)x^{j+k+i}\\
  &&-1^{\otimes3}\otimes_{H}k(\alpha(a)\cdot c)\cdot\alpha(b)x^{j+k+i+1}
  -1^{\otimes3}\otimes_{H}k(\alpha(a)\cdot c)\cdot\delta(b)x^{j+k+i}\\
  &&-(1\otimes 1\otimes f_t)\otimes_H\delta(a)\cdot m_tx^{j+k+i}\\
  &&+1^{\otimes3}\otimes_{H}jb\cdot\delta(a)\cdot\alpha(c)x^{j+k+i}+1^{\otimes3}\otimes_{H}jb\cdot\delta(a)\cdot\delta(c)x^{j+k+i-1}\\
  &&-1^{\otimes3}\otimes_{H}k(\delta(a)\cdot c)\cdot\alpha(b)x^{j+k+i}
  -1^{\otimes3}\otimes_{H}k(\delta(a)\cdot c)\cdot\delta(b)x^{j+k+i-1}.
\end{eqnarray*}
Therefore, we have
\begin{eqnarray*}
  && \{\{ax^i\ast bx^j\}\ast cx^{k+1}\}-\{ax^i\ast\{bx^j \ast cx^{k+1}\}\}+\big((\sigma\otimes\mathrm{id})\otimes_H \mathrm{id}\big)\{bx^j\ast\{ax^i \ast cx^{k+1}\}\} \\
  &=& \{\{ax^i\ast bx^j\}\ast cx^{k+1}\}+\big((\sigma\otimes\mathrm{id})\otimes_H \mathrm{id}\big)\big((\mathrm{id}\otimes\sigma)\otimes_H \mathrm{id}\big)\{\{bx^j \ast cx^{k+1}\}\ast ax^i\}\\
  &&-\big((\mathrm{id}\otimes\sigma)\otimes_H\mathrm{id}\big)\{\{ax^i \ast cx^{k+1}\}\ast bx^j\}\\
  &=&\{\{ax^i\ast bx^j\}\ast cx^{k}\}x+\{\{ax^{i}\ast bx^{j}\} \ast x\}cx^{k}\\
  &&+ \big((\sigma\otimes\mathrm{id})\otimes_H \mathrm{id}\big)\big((\mathrm{id}\otimes\sigma)\otimes_H \mathrm{id}\big)\{\{bx^j \ast cx^{k}\}\ast ax^i\}x\\
  &&-(1\otimes f_t\otimes1)\otimes_H m_t\cdot\alpha(a)x^{j+k+i+1}-(1\otimes f_t\otimes1)\otimes_H m_t\cdot\delta(a)x^{j+k+i}\\
  &&-1^{\otimes3}\otimes_{H}kc\cdot\alpha(b)\cdot\alpha(a)x^{j+k+i+1}-1^{\otimes3}\otimes_{H}kc\cdot\alpha(b)\cdot\delta(a)x^{j+k+i}\\
  &&+1^{\otimes3}\otimes_{H}jb\cdot\alpha(c)\cdot\alpha(a)x^{j+k+i+1}+1^{\otimes3}\otimes_{H}jb\cdot\alpha(c)\cdot\delta(a)x^{j+k+i}\\
  &&-1^{\otimes3}\otimes_{H}kc\cdot\delta(b)\cdot\alpha(a)x^{j+k+i}-1^{\otimes3}\otimes_{H}kc\cdot\delta(b)\cdot\delta(a)x^{j+k+i-1}\\
  &&+1^{\otimes3}\otimes_{H}jb\cdot\delta(c)\cdot\alpha(a)x^{j+k+i}+1^{\otimes3}\otimes_{H}jb\cdot\delta(c)\cdot\delta(a)x^{j+k+i-1}\\
  &&+\big((\sigma\otimes\mathrm{id})\otimes_H \mathrm{id}\big)\big((\mathrm{id}\otimes\sigma)\otimes_H \mathrm{id}\big)\{\{bx^j \ast x\}\ast ax^i\}cx^{k}\\
  &&+\big(1\otimes g_{s(1)(1)}\otimes g_{s(1)(2)}\big)\otimes_H(g_{s(-2)}\alpha(b))\cdot n_sx^{j+k+i+1}\\
  &&+1^{\otimes3}\otimes_{H}ia\cdot\alpha(b)\cdot\alpha(c)x^{j+k+i+1}+1^{\otimes3}\otimes_{H}ia\cdot\alpha(b)\cdot\delta(c)x^{j+k+i}\\
  &&-1^{\otimes3}\otimes_{H}k(\alpha(b)\cdot c)\cdot\alpha(a)x^{j+k+i+1}
  -1^{\otimes3}\otimes_{H}k(\alpha(b)\cdot c)\cdot\delta(a)x^{j+k+i}\\
  &&+\big(1\otimes g_{s(1)(1)}\otimes g_{s(1)(2)}\big)\otimes_H(g_{s(-2)}\delta(b))\cdot n_sx^{j+k+i}\\
  &&+1^{\otimes3}\otimes_{H}ia\cdot\delta(b)\cdot\alpha(c)x^{j+k+i}+1^{\otimes3}\otimes_{H}ia\cdot\delta(b)\cdot\delta(c)x^{j+k+i-1}\\
  &&-1^{\otimes3}\otimes_{H}k(\delta(b)\cdot c)\cdot\alpha(a)x^{j+k+i}
  -1^{\otimes3}\otimes_{H}k(\delta(b)\cdot c)\cdot\delta(a)x^{j+k+i-1}\\
  &&-\big((\mathrm{id}\otimes\sigma)\otimes_H\mathrm{id}\big)\{\{ax^i \ast cx^{k}\}\ast bx^j\}x\\
  &&-\big((1\otimes 1\otimes g_s)\otimes_H n_s\cdot\alpha(b)\big)x^{j+k+i+1}-\big((1\otimes 1\otimes g_s)\otimes_H n_s\cdot\delta(b)\big)x^{j+k+i}\\
  &&+1^{\otimes3}\otimes_{H}kc\cdot\alpha(a)\cdot\alpha(b)x^{j+k+i+1}+1^{\otimes3}\otimes_{H}kc\cdot\alpha(a)\cdot\delta(b)x^{j+k+i}\\
  &&-1^{\otimes3}\otimes_{H}ia\cdot\alpha(c)\cdot\alpha(b)x^{j+k+i+1}-1^{\otimes3}\otimes_{H}ia\cdot\alpha(c)\cdot\delta(b)x^{j+k+i}\\
  &&+1^{\otimes3}\otimes_{H}kc\cdot\delta(a)\cdot\alpha(b)x^{j+k+i}+1^{\otimes3}\otimes_{H}kc\cdot\delta(a)\cdot\delta(b)x^{j+k+i-1}\\
  &&-1^{\otimes3}\otimes_{H}ia\cdot\delta(c)\cdot\alpha(b)x^{j+k+i}-1^{\otimes3}\otimes_{H}ia\cdot\delta(c)\cdot\delta(b)x^{j+k+i-1}\\
  &&-\big((\mathrm{id}\otimes\sigma)\otimes_H\mathrm{id}\big)\{\{ax^i \ast x\}\ast bx^j\}cx^{k}\\
  &&+(1\otimes f_t\otimes 1)\otimes_H\alpha(a)\cdot m_tx^{j+k+i+1}+(1\otimes f_t\otimes 1)\otimes_H\delta(a)\cdot m_tx^{j+k+i}\\
  &&-1^{\otimes3}\otimes_{H}jb\cdot\alpha(a)\cdot\alpha(c)x^{j+k+i+1}-1^{\otimes3}\otimes_{H}jb\cdot\alpha(a)\cdot\delta(c)x^{j+k+i}\\
  &&+1^{\otimes3}\otimes_{H}k(\alpha(a)\cdot c)\cdot\alpha(b)x^{j+k+i+1}
  +1^{\otimes3}\otimes_{H}k(\alpha(a)\cdot c)\cdot\delta(b)x^{j+k+i}\\
  &&-1^{\otimes3}\otimes_{H}jb\cdot\delta(a)\cdot\alpha(c)x^{j+k+i}-1^{\otimes3}\otimes_{H}jb\cdot\delta(a)\cdot\delta(c)x^{j+k+i-1}\\
  &&+1^{\otimes3}\otimes_{H}k(\delta(a)\cdot c)\cdot\alpha(b)x^{j+k+i}
  +1^{\otimes3}\otimes_{H}k(\delta(a)\cdot c)\cdot\delta(b)x^{j+k+i-1}\\
  &=&\big(\{\{ax^i\ast bx^j\}\ast cx^{k}\}+\big((\sigma\otimes\mathrm{id})\otimes_H \mathrm{id}\big)\big((\mathrm{id}\otimes\sigma)\otimes_H \mathrm{id}\big)\{\{bx^j \ast cx^{k}\}\ast ax^i\}\\
  &&-\big((\mathrm{id}\otimes\sigma)\otimes_H\mathrm{id}\big)\{\{ax^i \ast cx^{k}\}\ast bx^j\}\big)x\\
  &&+\big(\{\{ax^i\ast bx^j\}\ast x\}+\big((\sigma\otimes\mathrm{id})\otimes_H \mathrm{id}\big)\big((\mathrm{id}\otimes\sigma)\otimes_H \mathrm{id}\big)\{\{bx^j \ast x\}\ast ax^i\}\\
  &&-\big((\mathrm{id}\otimes\sigma)\otimes_H\mathrm{id}\big)\{\{ax^i \ast x\}\ast bx^j\}\big)cx^{k}\\
  &&+\big(1\otimes g_{s(1)(1)}\otimes g_{s(1)(2)}\big)\otimes_H(g_{s(-2)}\alpha(b))\cdot n_sx^{j+k+i+1}\\
  &&+\big(1\otimes g_{s(1)(1)}\otimes g_{s(1)(2)}\big)\otimes_H(g_{s(-2)}\delta(b))\cdot n_sx^{j+k+i}\\
  &&-\big((1\otimes 1\otimes g_s)\otimes_H n_s\cdot\alpha(b)\big)x^{j+k+i+1}-\big((1\otimes 1\otimes g_s)\otimes_H n_s\cdot\delta(b)\big)x^{j+k+i}\\
  &=&\big(\{\{ax^i\ast bx^j\}\ast cx^{k}\}+\big((\sigma\otimes\mathrm{id})\otimes_H \mathrm{id}\big)\big((\mathrm{id}\otimes\sigma)\otimes_H \mathrm{id}\big)\{\{bx^j \ast cx^{k}\}\ast ax^i\}\\
  &&-\big((\mathrm{id}\otimes\sigma)\otimes_H\mathrm{id}\big)\{\{ax^i \ast cx^{k}\}\ast bx^j\}\big)x\\
  &&+\big(\{\{ax^i\ast bx^j\}\ast x\}+\big((\sigma\otimes\mathrm{id})\otimes_H \mathrm{id}\big)\big((\mathrm{id}\otimes\sigma)\otimes_H \mathrm{id}\big)\{\{bx^j \ast x\}\ast ax^i\}\\
  &&-\big((\mathrm{id}\otimes\sigma)\otimes_H\mathrm{id}\big)\{\{ax^i \ast x\}\ast bx^j\}\big)cx^{k}\\
  &&+\big(1\otimes g_{s(1)}\otimes g_{s(2)}\big)\otimes_H(g_{s(-3)}\alpha(b))\cdot n_sx^{j+k+i+1}
  +\big(1\otimes g_{s(1)}\otimes g_{s(2)}\big)\otimes_H(g_{s(-3)}\delta(b))\cdot n_sx^{j+k+i}\\
  &&-\big((1\otimes g_{s(1)}\otimes g_{s(2)}\big)\otimes_H n_s\cdot\alpha(b)(g_{s(-3)}x^{j+k+i+1})-\big(1\otimes g_{s(1)}\otimes g_{s(2)}\big)\otimes_H n_s\cdot\delta(b)(g_{s(-3)}x^{j+k+i})\\
  &=&0.
\end{eqnarray*}

Hence the conclusion is completed.
\end{proof}
The Poisson $H$-pseudoalgebra $A[x]$ endowed with the Poisson pseudobracket from Theorem \ref{thm:Poisson-pseudo-poly-ring} is denoted by $A[x;\alpha,\delta]_{PP}$ and called {\bf Poisson pseudo-Ore extension} of $(A,\cdot,\{\ast\}_A)$.

\section{Some examples of Poisson $H$-pseudoalgebras}\label{sec:examples}

The notions of current $H$-pseudoalgebras and annihilation algebras of $H$-pseudoalgebras have been proposed in \cite{Bakalov-2001-1}. This section builds on that work and aims to gives some examples of Poisson $H$-pseudoalgebras.  

\subsection{current $H$-pseudoalgebras}\label{subsec:current}
\begin{ex}\label{ex:current}
$($\cite[Section 4]{Bakalov-2001-1}$)$
Let $H'$ be a Hopf subalgebra of $H$ and $(A,\ast_A)$ be an $H'$-pseudoalgebra, for all $a,b\in A$ and $a\ast_A b=\sum\limits_{i}(f_i\otimes g_i)\otimes_{H'} e_i$, if we define
\begin{equation*}
  (f\otimes_{H'} a)\ast (g\otimes_{H'} b)=\sum\limits_{i}(ff_i\otimes gg_i)\otimes_H(1\otimes_{H'} e_i),
\end{equation*}
then $(Cur A:~=H\otimes_{H'} A,\ast)$ is an $H$-pseudoalgebra which is Lie or associative when $(A,\ast_A)$ is so.
$(Cur A,\ast)$ is called a {\bf current pseudoalgebra} of $(A,\ast_A)$.
\end{ex}
Specially, when $H'=\mathbf{k}$, let $(A,\{\cdot,\cdot\})$ be a Lie algebra. For all $a, b\in A$ and $f, g\in H$, if we define
\begin{equation}\label{eq:current-Lie}
  \{(f\otimes a)\ast(g\otimes b)\}=(f\otimes g)\otimes_H(1\otimes\{a,b\}),
\end{equation}
then $(Cur A:~=H\otimes A, \{\ast\})$ is a Lie $H$-pseudoalgebra. 


\begin{ex}
Let $(A,\cdot,\{\cdot,\cdot\})$ be a commutative Poisson algebra and $(Cur A=H\otimes_H A, \{\ast\})$ be a Lie $H$-pseudoalgebra defined by \eqref{eq:current-Lie}. For all $f, g, h \in H$ and $a, b\in A$, let algebraic structure in $H$ be commutative, i. e., $gh=hg$, if we define
\begin{eqnarray*}
   h(fg)&:=& (h_{(1)}f)(h_{(2)}g),\\
   (f\otimes gh)\otimes_H(1\otimes a)&:=&(f\otimes g_{(1)})\otimes_H g_{(-2)}h(1\otimes a) \\
  (f\otimes_H a)\bullet(g\otimes_H b) &:=& fg\otimes_H (a\cdot b), \\
  \big((f\otimes g)\otimes_H (1\otimes a)\big)\bullet(h\otimes_H b)&:=& 
  (f\otimes g_{(1)})\otimes_H (g_{(-2)}h\otimes_H a\cdot b),
\end{eqnarray*}
then $(Cur A, \bullet, \{\ast\})$ is a Poisson $H$-pseudoalgebra.
\end{ex}

\begin{proof}
It suffices to prove that $(Cur A, \bullet, \{\ast\})$ satisfies $H$-differential \eqref{eq:H-differential} and left Leibniz rule \eqref{eq:left-Leibniz-rule}.

For all $f, g, h \in H$ and $a, b, c\in A$, $H$-differential is true since
\begin{eqnarray*}
  h\big((f\otimes_H a)\bullet(g\otimes_H b)\big) &=& h\big(fg\otimes_H (a\cdot b)\big)=h(fg)\otimes_H(a\cdot b) 
  =(h_{(1)}f)(h_{(2)}g)\otimes_H (a\cdot b)\\
  &=&(h_{(1)}f\otimes_Ha)\bullet(h_{(2)}g\otimes_Hb)=\big(h_{(1)}(f\otimes_H a)\big)\bullet\big(h_{(2)}(g\otimes_H b)\big),
\end{eqnarray*}
left Leibniz rule follows from
\begin{eqnarray*}
  &&\{(f\otimes_H a)\ast\big((g\otimes_H b)\bullet (h\otimes_H c)\big)\} = \{(f\otimes_H a)\ast \big(gh\otimes_H (b\cdot c)\big)\}
  =(f\otimes gh)\otimes_H\big(1\otimes\{a,b\cdot c\}\big) \\
  &=& (f\otimes gh)\otimes_H\big(1\otimes\{a,b\}\cdot c\big)+(f\otimes gh)\otimes_H\big(1\otimes b\cdot\{a,c\}\big)\\
  &=& (f\otimes gh)\otimes_H\big(1\otimes\{a,b\}\cdot c\big)+(f\otimes hg)\otimes_H\big(1\otimes\{a,c\}\cdot b\big)\\
  &=&(f\otimes g_{(1)})\otimes_H \big(g_{(-2)}h\otimes\{a,b\}\cdot c\big)
  +(f\otimes h_{(1)})\otimes_H \big(h_{(-2)}g\otimes_H\{a,c\}\cdot b\big)\\
  &=&\big((f\otimes g)\otimes_H(1\otimes\{a,b\})\big)\bullet(h\otimes_H c)+\big((f\otimes h)\otimes_H(1\otimes\{a,c\})\big)\bullet(g\otimes_H b)\\
  &=&\{(f\otimes_H a)\ast(g\otimes_H b)\}\bullet(h\otimes_H c)+\{(f\otimes_H a)\ast(h\otimes_H c)\}\bullet(g\otimes_H b).
\end{eqnarray*}
\end{proof}

\subsection{$H$-differential algebras and annihilation algebras}\label{subsec:differ-annihi}

\begin{defi}
$($\cite[Section 2]{Bakalov-2001-1}$)$
Let $(A,\cdot)$ be an associative algebra, for all $h\in H$ and $a, b\in A$, if $A$ is a left $H$-module satisfies
\begin{equation}\label{eq:left-H-differ-alg}
  h(a\cdot b)=(h_{(1)}a)\cdot(h_{(2)}b),
\end{equation}
then $(A,\cdot)$ is called an {\bf $H$-differential algebra}.
\end{defi}
Similarly, the right action of $H$ satisfies
\begin{equation}\label{eq:right-H-differ-alg}
  (a\cdot b)h=(ah_{(1)})\cdot(bh_{(2)}),\quad\forall h\in H,~a, b\in A,
\end{equation}

Let $Y$ be an $H$-bimodule and $(Y,\cdot)$ be a commutative associative $H$-differential algebra satisfying \eqref{eq:left-H-differ-alg} and \eqref{eq:right-H-differ-alg}. Let $L$ be a left $H$-module and $A_YL:~=Y\otimes_H L$. A left action of $H$ on $A_YL$ is defined by
\begin{equation*}
  h(x\otimes_H a)=hx\otimes_H a,\quad\forall h\in H,~x\in Y,~a\in L.
\end{equation*}
Let $(L,\ast)$ be an $H$-pseudoalgebra, for all $x, y\in Y,~a, b\in L$ and $a\ast b=\sum\limits_i(f_i\otimes g_i)\otimes_H e_i$, a product on $A_YL$ is defined by
\begin{equation*}
  (x\otimes_H a)(y\otimes_H b)=\sum\limits_i(xf_i)\cdot(yg_i)\otimes_H e_i.
\end{equation*}

Specially, when $Y$ is the dual algebra of $H$, it induces the following notions of the annihilation algebras.
\begin{defi}
$($\cite[Section 7]{Bakalov-2001-1}$)$
Let $(L,\ast)$ be an $H$-pseudoalgebra and $X:~=H^{\ast}$ is a dual algebra of $H$, the $H$-differential algebra $A_XL:~=X\otimes_H L$ is called an {\bf annihilation algebra} of $(L,\ast)$. 
\end{defi}

\begin{defi}
Let $(L,\{\ast\})$ be a Lie $H$-pseudoalgebra and $X:~=H^{\ast}$ is a dual algebra of $H$, the $H$-differential algebra $A_XL:~=X\otimes_H L$ is called an {\bf annihilation Lie algebra} of $(L,\{\ast\})$, where Lie bracket $\{\cdot,\cdot\}$ on $A_XL$ is defined as follows:
\begin{equation}\label{eq:annihi-Lie-alg}
  \{x\otimes_H a, y\otimes_H b\}=\sum\limits_i(xf_i)(yg_i)\otimes_H e_i,
\end{equation}
for all $a, b\in L,~x, y\in X$, and $\{a\ast b\}=\sum\limits_i(f_i\otimes g_i)\otimes_H e_i$.

\end{defi}

\begin{pro}\label{pro:annihi-Poi-alg}
Let $(L,\cdot,\{\ast\})$ be a commutative Poisson $H$-pseudoalgebra and $A_XL:~=X\otimes_H L$ be an annihilation Lie algebra of $(L,\{\ast\})$, where Lie bracket $\{\cdot,\cdot\}$ on $A_XL$ is defined by \eqref{eq:annihi-Lie-alg}. For all $f,g \in H,~x,y,z\in X$ and $a, b, c \in L$, if we define
\begin{eqnarray*}
  (x\otimes_{H} a)\bullet (y\otimes_{H} b) &:=& xy\otimes_{H} (a\cdot b), \\
 \big((xf)(yg)\otimes_H a\big)\bullet(z\otimes_H c)&:=& (xf)((zy)g_{(1)})\otimes_H a\cdot(g_{(-2)}c),
\end{eqnarray*}
then $(A_XL,\bullet,\{\cdot,\cdot\})$ is a commutative Poisson algebra.
\end{pro}

\begin{proof}
The commutativity of associative algebra $(A_XL,\bullet)$ is obvious.
For $\{a\ast b\}=\sum\limits_i(f_i\otimes g_i)\otimes_H e_i$, $\{a\ast c\}=\sum\limits_j(h_j\otimes l_j)\otimes_H d_j$, by \eqref{eq:left-Leibniz-rule} and \eqref{eq:left-Leibniz}, we have
\begin{equation*}
  \{a\ast (b\cdot c)\}=\sum\limits_i(f_i\otimes g_{i(1)})\otimes_H e_i\cdot(g_{i(-2)}c)+\sum\limits_j(h_j\otimes l_{j(1)})\otimes_H d_j\cdot(l_{j(-2)}b),
\end{equation*}
by \eqref{eq:annihi-Lie-alg}, we have
\begin{eqnarray*}
  &&\{x\otimes_H a, (y\otimes_{H} b)\bullet (z\otimes_{H} c)\} = \{x\otimes_H a, yz\otimes_{H} (b\cdot c)\} \\
  &&=  \sum\limits_i(xf_i)((yz)g_{i(1)})\otimes_H e_i\cdot(g_{i(-2)}c)+\sum\limits_j(xh_j)((yz)l_{j(1)})\otimes_H d_j\cdot(l_{j(-2)}b)\\
  &&=\sum\limits_i(xf_i)((zy)g_{i(1)})\otimes_H e_i\cdot(g_{i(-2)}c)+\sum\limits_j(xh_j)((yz)l_{j(1)})\otimes_H d_j\cdot(l_{j(-2)}b)\\
  &&=\sum\limits_i\big((xf_i)(yg_i)\otimes_H e_i\big) \bullet(z\otimes_H c)
  +\sum\limits_j\big((xh_j)(zl_j)\otimes_H d_j\big) \bullet(y\otimes_H b) \\
  &&=\{x\otimes_H a, y\otimes_H b\}\bullet(z\otimes_H c)+\{x\otimes_H a, z\otimes_H c\}\bullet(y\otimes_H b) \\
  &&=\{x\otimes_H a, y\otimes_H b\}\bullet(z\otimes_H c)+(y\otimes_H b)\bullet \{x\otimes_H a, z\otimes_H c\}. 
\end{eqnarray*}
Therefore, the proof is completed.

\end{proof}

\section{DGP pseudoalgebras and their universal enveloping $H$-pseudoalgebras}\label{sec:DGPpa-Uni}

We consider differential graded Poisson algebras and their universal enveloping algebras defined in \cite[Section 3$\sim$4]{Lv-2016-4} in a pseudotensor category $\mathcal{M}^{\ast}(H)$, inducing the notions of differential graded Poisson $H$-pseudoalgebras and their universal enveloping $H$-pseudoalgebras.
We consider the Poisson pseudobracket of degree $p$.

\subsection{DGP pseudoalgebras}\label{sec:DGPpa}

\begin{defi}
Let $A=\bigoplus_{i\in \mathbb{Z}} A_i$ be a graded left $H$-module, $\ast:~A\otimes A\to (H\otimes H)\otimes_H A$ be an $H$-bilinear map and $d:~A\rightarrow A$ be an $H$-linear map of degree $1$ satisfying
\begin{enumerate} 
\item (differential) $d^2=0$; and 
\item $\big((\mathrm{id}\otimes\mathrm{id})\otimes_H d\big)(a\ast b)=d(a)\ast b+(-1)^{|a|}a\ast d(b)$, for all homogeneous elements $a, b\in A$, 
\end{enumerate} 
then $(A, \ast, d)$ is called a {\bf differential graded $H$-pseudoalgebra} (or DG pseudoalgebra).
Moreover, if $a\ast b=(-1)^{|a||b|}(\sigma\otimes_H\mathrm{id})(b\ast a)$, for all homogeneous elements $a, b\in A$, then DG pseudoalgebra $(A, \ast, d)$ is called {\bf graded commutative}.
\end{defi}

\begin{defi}
Let $A=\bigoplus_{i\in \mathbb{Z}} A_i$ be a graded left $H$-module, $\{\ast\}:~A\otimes A\to (H\otimes H)\otimes_H A$ be an $H$-bilinear map of degree $p$ and $d:~A\to A$ be an $H$-linear map of degree $1$ ($d^2=0$) such that 
\begin{enumerate} 
\item (skew-symmetry) $\{a\ast b\}=-(-1)^{(|a|+p)(|b|+p)}(\sigma\otimes_H\mathrm{id})\{b\ast a\}$; 
\item (Jacobi identity) $\{a\ast\{b\ast c\}\} -(-1)^{(|a|+p)(|b|+p)}\big((\sigma\otimes\mathrm{id})\otimes_H\mathrm{id}\big)\{b\ast\{a\ast c\}\}=\{\{a\ast b\}\ast c\}$; and 
\item $\big((\mathrm{id}\otimes\mathrm{id})\otimes_H d\big)\{a\ast b\}=\{d(a)\ast b\}+(-1)^{(|a|+p)}\{a\ast d(b)\}$, for all homogeneous elements $a, b, c\in A$, 
\end{enumerate} 
then $(A, \{\ast\}, d)$ is called a {\bf differential graded Lie $H$-pseudoalgebra} (or DGL pseudoalgebra).
\end{defi}

\begin{defi}
Let $(A,\cdot,d)$ be a differential graded algebra and $(A,\{\ast\},d)$ be a DGL pseudoalgebra satisfying $H$-differential \eqref{eq:H-differential} and {\bf graded left Leibniz rule}
\begin{equation}\label{eq:graded-left-Leibniz-rule}
  \{a\ast (b\cdot c)\}=\{a\ast b\}\cdot c+(-1)^{(|c|+p)|b|}\{a\ast c\}\cdot b,
\end{equation}
for all homogeneous elements $a, b, c\in A$, then $(A, \cdot, \{\ast\}, d)$ is called a {\bf differential graded Poisson $H$-pseudoalgebra} (or DGP pseudoalgebra).
Moreover, if $(A,\cdot,d)$ is graded commutative, then $(A,\cdot,\{\ast\}, d)$ is a {\bf graded commutative DGP pseudoalgebra}.
\end{defi}

\begin{lem}
Let $(A,\cdot,\{\ast\}, d)$ be a DGP pseudoalgebra, for all homogeneous elements $a, b, c\in A$, similar to right Leibniz rule \eqref{eq:right-Leibniz-rule} in Lemma \ref{lem:right-Leibniz-rule}, then we have the {\bf graded right Leibniz rule}
\begin{equation}\label{eq:graded-right-Leibniz-rule}
  \{(a\cdot b)\ast c\}=a\cdot\{b\ast c\}+(-1)^{(|a|+p)|b|}b\cdot\{a\ast c\}.
\end{equation}
\end{lem}

\begin{defi}
Let $(A,\cdot,\{\ast\},d)$ be a DGP pseudoalgebra, and $(B,\cdot)$ is a subalgebra of $(A,\cdot)$.
\begin{enumerate}
  \item If $d(B)\subseteq B$ and $\{B\ast B\}\subseteq (H\otimes H)\otimes_HB$, then $(B,\cdot,\{\ast\},d)$ is called a DGP 
  {\bf pseudosubalgebra} of $(A,\cdot,\{\ast\},d)$, denoted by $B<A$.
  \item If $d(B)\subseteq B, A\cdot B\subseteq B, B\cdot A\subseteq B$ and $\{A\ast B\}, \{B\ast A\}$ are all contained in $(H\otimes H)\otimes_HB$, then $(B,\cdot,\{\ast\},d)$ is called a DGP {\bf pseudoideal} of $(A,\cdot,\{\ast\},d)$, denoted by $B\lhd A$.
\end{enumerate}

\end{defi}

\begin{defi}\label{defi:DGP-homo}
Let $(A,\cdot_A,\{\ast\}_A,d_A)$ and $(B,\cdot_B,\{\ast\}_B,d_B)$ be two DGP pseudoalgebras. An $H$-linear map $f:~A\rightarrow B$ is called a DGP pseudoalgebra {\bf homomorphism} if $f(a\cdot_Ab)=f(a)\cdot_Bf(b)$, $\big((\mathrm{id}\otimes\mathrm{id})\otimes_H f\big)\{a\ast b\}_A=\{f(a)\ast f(b)\}_B$ and $f\circ d_A=d_B\circ f$ for all homogeneous elements $a, b\in A$. Moreover, $f$ is called a DGP pseudoalgebra {\bf monomorphism} if $\mathrm{Ker}f=0$, a DGP pseudoalgebra {\bf epimorphism} if $\mathrm{Im}f=B$, and a DGP pseudoalgebra {\bf isomorphism} if it is a bijection (in this case, we write $A\cong B$ to denote the isomorphism between $A$ and $B$).
\end{defi}

%

\begin{pro}\label{pro:oppo-psalg}
Let $(A,\cdot,\{\ast\},d)$ be a DGP pseudoalgebra, then $(A^{op},\cdot_{op},\{\ast\}_{op},d_{op})$ is a DGP pseudoalgebra, where for all homogeneous elements $a,b \in A$,
\begin{eqnarray}
  \label{eq:asso-op}a\cdot_{op}b &:=& a\cdot b, \\
  \label{eq:diff-op}d_{op} &:=& d, \\
  \nonumber \{a\ast b\}_{op} &:=& -\{a\ast b\}. 
\end{eqnarray}

\end{pro}

\begin{proof}
By \eqref{eq:asso-op} and \eqref{eq:diff-op}, $(A^{op},\cdot_{op},d_{op})$ is a differential graded algebra and satisfies $H$-differential. 

Observe that:
\begin{equation*}
  \{a\ast b\}_{op} = -\{a\ast b\}=(-1)^{(|a|+p)(|b|+p)}(\sigma\otimes_H\mathrm{id})\{b\ast a\}=-(-1)^{(|a|+p)(|b|+p)}(\sigma\otimes_H\mathrm{id})\{b\ast a\}_{op},
\end{equation*}
\begin{eqnarray*}
  \{a\ast\{b\ast c\}_{op}\}_{op} &=& \{a\ast\{b\ast c\}\}
  =\{\{a\ast b\}\ast c\}+(-1)^{(|a|+p)(|b|+p)}\big((\sigma\otimes\mathrm{id})\otimes_H\mathrm{id}\big)\{b\ast\{a\ast c\}\} \\
  &=& \{\{a\ast b\}_{op}\ast c\}_{op}+(-1)^{(|a|+p)(|b|+p)}\big((\sigma\otimes\mathrm{id}\big)\otimes_H\mathrm{id})\{b\ast\{a\ast c\}_{op}\}_{op},
\end{eqnarray*}
and
\begin{eqnarray*}
  \big((\mathrm{id}\otimes\mathrm{id})\otimes_H d_{op}\big)\{a\ast b\}_{op} &=& -\big((\mathrm{id}\otimes\mathrm{id})\otimes_H d\big)\{a\ast b\}=-\{d(a)\ast b\}-(-1)^{|a|+p}\{a\ast d(b)\} \\
  &=& \{d_{op}(a)\ast b\}_{op}+(-1)^{|a|+p}\{a\ast d_{op}(b)\}_{op},
\end{eqnarray*}
hence $(A^{op},\{\cdot,\cdot\}_{op},d_{op})$ is a DGL pseudoalgebra.
Furthermore, we have
\begin{eqnarray*}
  \{a\ast (b\cdot_{op}c)\}_{op} &=& -\{a\ast (b\cdot c)\} 
  =-\{a\ast b\}\cdot c-(-1)^{(|c|+p)|b|}\{a\ast c\}\cdot b \\
  &=& \{a\ast b\}_{op}\cdot_{op}c+(-1)^{(|c|+p)|b|}\{a\ast c\}_{op}\cdot_{op}b.
\end{eqnarray*}
Thus, the conclusion holds.
\end{proof}

\begin{pro}\label{pro:tensor-psalg}
Let $(A,\cdot_A,\{\ast\}_A,d_A)$ and $(B,\cdot_B,\{\ast\}_B,d_B)$ be two graded commutative DGP pseudoalgebras with pseudobrackets $\{\ast\}_A$ and $\{\ast\}_B$ of degree $0$. For all homogeneous elements $a, a' \in A$ and $b, b' \in B$, suppose that $\{a\ast a'\}_A=\sum_i(u_i\otimes v_i)\otimes_H\alpha_i$ and $\{b\ast b'\}_B=\sum_j(u_j\otimes v_j)\otimes_H\beta_j$, set
\begin{eqnarray*}
  &&h(a\otimes b):=h_{(1)}a\otimes h_{(2)}b,\\
  &&(a\otimes b) \star (a'\otimes b'):=(-1)^{|a'||b|}(a\cdot_A a')\otimes(b\cdot_B b'), \\
  &&d(a\otimes b):=d_A(a)\otimes b+(-1)^{|a|}a\otimes d_B(b),\\
  &&\{(a\otimes b)\ast (a'\otimes b')\}:=(-1)^{|a'||b|}\big(\{a\ast a'\}_A\otimes (b\cdot_B b')+(a\cdot_A a')\otimes\{b\ast b'\}_B\big)\\
  &=&(-1)^{|a'||b|}\big((u_i\otimes v_i)\otimes_H(\alpha_i\otimes(b\cdot_B b'))+(u_j\otimes v_j)\otimes_H((a\cdot_A a')\otimes\beta_j)\big),
\end{eqnarray*}
satisfying the following compatible conditions:
\begin{eqnarray}
  \label{eq:tensor-com1}&&(1\otimes h\otimes1)(1\otimes\Delta)(1\otimes g_{(1)})\otimes_H
  \big((a\cdot_Aa')\otimes(b\cdot_B(g_{(-2)}b'))\big)\\
   \nonumber&=&(1\otimes g\otimes1)(\Delta\otimes1)(h_{(1)}\otimes1)\otimes_H\big(((h_{(-2)}a)\cdot_Aa')\otimes(b\cdot_Bb')\big), \\
   \label{eq:tensor-com2}&&(1\otimes h\otimes f)(1\otimes\Delta)(1\otimes g_{(1)})\otimes_H
   \big((a\cdot_Aa')\otimes (b\cdot_B(g_{(-2)}b'))\big)\\
   \nonumber&=&\big((\mathrm{id}\otimes\sigma)\otimes_H\mathrm{id}_{A\otimes B}\big)
   (1\otimes g\otimes h)(\Delta\otimes1)(f_{(1)}\otimes1)\otimes_H\big(((f_{(-2)}a)\cdot_Aa')\otimes(b\cdot_Bb')\big), \\
   \label{eq:tensor-com3}&&
   (f\otimes h\otimes1)(\Delta\otimes1)(g_{(1)}\otimes1)\otimes_H\big(((g_{(-2)}a)\cdot_Aa')\otimes(b\cdot_Bb')\big)\\
   \nonumber&=&\big((\sigma\otimes\mathrm{id})\otimes_H\mathrm{id}_{A\otimes B}\big)(h\otimes g\otimes1)(1\otimes\Delta)(1\otimes f_{(1)})\otimes_H\big((a\cdot_Aa')\otimes(b\cdot_B(f_{(-2)}b'))\big),\\
  \label{eq:tensor-com4}&&a\cdot_A(ha')\otimes b\cdot_Bb'=a\cdot_Aa'\otimes b\cdot_B(hb')=a\cdot_A(h_{(1)}a')\otimes b\cdot_B(h_{(2)}b'), 
\end{eqnarray}
for all $f, g, h\in H$, where the identity map of $A\otimes B$ is denoted by $\mathrm{id}_{A\otimes B}$, then $(A\otimes B,\star,\{\ast\},d)$ is a graded commutative DGP pseudoalgebra.
\end{pro}

\begin{proof}
For all homogeneous elements $a\otimes b,a'\otimes b',a''\otimes b''\in A\otimes B$, we have
\begin{eqnarray*}
  (a'\otimes b') \star (a\otimes b) &=& (-1)^{|a||b'|}(a'\cdot_A a)\otimes(b'\cdot_B b)  \\
  &=& (-1)^{|a||b'|+|a||a'|+|b||b'|}(a\cdot_A a')\otimes(b\cdot_B b') \\
  &=& (-1)^{(|a|+|b|)(|a'|+|b'|)}(a\otimes b) \star (a'\otimes b'), 
\end{eqnarray*}
and
\begin{eqnarray*}
  &&d\big((a\otimes b) \star (a'\otimes b')\big) = (-1)^{|a'||b|}d\big((a\cdot_A a')\otimes(b\cdot_B b')\big)  \\
  &=& (-1)^{|a'||b|}\big(d_A(a\cdot_A a')\otimes(b\cdot_B b')+(-1)^{|a|+|a'|}(a\cdot_A a')\otimes d_B(b\cdot_B b')\big) \\
  &=& (-1)^{|a'||b|}(d_A(a)\cdot_A a')\otimes(b\cdot_B b')+(-1)^{|a'||b|+|a|}(a\cdot_A d_A(a'))\otimes(b\cdot_B b')\\
  &&+(-1)^{|a'||b|+|a|+|a'|}(a\cdot_A a')\otimes (d_B(b)\cdot_B b')+(-1)^{|a'||b|+|a|+|a'|+|b|}(a\cdot_A a')\otimes (b\cdot_B d_B(b'))\\
  &=&\big(d(a)\otimes b\big) \star (a'\otimes b')+(-1)^{|a|}\big(a\otimes d_B(b)\big) \star (a'\otimes b')\\
  &&+(-1)^{|a|+|b|}(a\otimes b) \star \big(d_A(a')\otimes b'\big)+(-1)^{|a|+|b|+|a'|}(a\otimes b) \star \big(a'\otimes d_B(b')\big)\\
  &=&d(a\otimes b) \star (a'\otimes b')+(-1)^{|a|+|b|}(a\otimes b) \star d(a'\otimes b'),
\end{eqnarray*}
thus, $(A\otimes B,\star,d)$ is a graded commutative differential graded algebra.

Furthermore, we have
\begin{eqnarray*}
  &&-(-1)^{(|a|+|b|)(|a'|+|b'|)}(\sigma\otimes_H \mathrm{id})\{(a'\otimes b') \ast (a\otimes b)\} \\
  &=& (-1)^{|a'||b|}\big((u_i\otimes v_i)\otimes_H(\alpha_i\otimes(b\cdot_B b'))+(u_j\otimes v_j)\otimes_H((a\cdot_A a')\otimes\beta_j)\big)\\
  &=& \{(a\otimes b) \ast (a'\otimes b')\},
\end{eqnarray*}
for all homogeneous elements $a,c,e\in A$ and $b,d,f\in B$, suppose that
\begin{eqnarray*}
  && \{a\ast c\}_A=(h_i\otimes l_i)\otimes_H a_i,
  ~\{c\ast e\}_A=(h_j\otimes l_j)\otimes_H a_j,
  ~\{a\ast e\}_A=(h_m\otimes l_m)\otimes_H a_m,\\
  && \{b\ast d\}_B=(f_i\otimes g_i)\otimes_H b_i,
  ~\{d\ast f\}_B=(f_j\otimes g_j)\otimes_H b_j,
  ~\{b\ast f\}_B=(f_m\otimes g_m)\otimes_H b_m,\\
  && \{a_i\ast e\}_A=(h_{ik}\otimes l_{ik})\otimes_H a_{ik},
  ~\{a\ast a_j\}_A=(h_{jk}\otimes l_{jk})\otimes_H a_{jk},
  ~\{c\ast a_m\}_A=(h_{mk}\otimes l_{mk})\otimes_H a_{mk},\\
  && \{b_i\ast f\}_B=(f_{ik}\otimes g_{ik})\otimes_H b_{ik},
  ~\{b\ast b_j\}_B=(f_{jk}\otimes g_{jk})\otimes_H b_{jk},
  ~\{d\ast b_m\}_B=(f_{mk}\otimes g_{mk})\otimes_H b_{mk},
\end{eqnarray*}
by \eqref{eq:left-Leibniz}, \eqref{eq:right-Leibniz}, \eqref{eq:graded-left-Leibniz-rule} and \eqref{eq:graded-right-Leibniz-rule}, we have
\begin{eqnarray*}
  && \{a\ast (c\cdot_A e)\}_A=\{a\ast c\}_A\cdot_A e+(-1)^{|e||c|}\{a\ast e\}_A\cdot_A c\\
  &=&(h_i\otimes l_{i(1)})\otimes_H a_i\cdot_A(l_{i(-2)}e)+(-1)^{|e||c|}(h_m\otimes l_{m(1)})\otimes_H a_m\cdot_A(l_{m(-2)}c),\\
  && \{(a\cdot_A c)\ast e\}_A=a\cdot_A\{c\ast e\}_A+(-1)^{|a||c|}c\cdot_A\{a\ast e\}_A\\
  &=&(h_{j(1)}\otimes l_j)\otimes_H(h_{j(-2)}a)\cdot_A a_j+(-1)^{|a||c|}(h_{m(1)}\otimes l_m)\otimes_H(h_{m(-2)}c)\cdot_A a_m,\\
  && \{c\ast (a\cdot_A e)\}_A=\{c\ast a\}_A\cdot_A e+(-1)^{|a||e|}\{c\ast e\}_A\cdot_A a\\
  &=&-(-1)^{|c||a|}(l_i\otimes h_{i(1)})\otimes_H a_i\cdot_A(h_{i(-2)}e)+(-1)^{|a||e|}(h_j\otimes l_{j(1)})\otimes_H a_j\cdot_A(l_{j(-2)}a).\\
\end{eqnarray*}
Similarly, we have
\begin{eqnarray*}
  && \{b\ast (d\cdot_B f)\}_B=\{b\ast d\}_B\cdot_B f+(-1)^{|d||f|}\{b\ast f\}_B\cdot_B d\\
  &=&(f_i\otimes g_{i(1)})\otimes_H b_i\cdot_B(g_{i(-2)}f)+(-1)^{|d||f|}(f_m\otimes g_{m(1)})\otimes_H b_m\cdot_B(g_{m(-2)}d),\\
  && \{(b\cdot_B d)\ast f\}_B=b\cdot_B\{d\ast f\}_B+(-1)^{|b||d|}d\cdot_B\{b\ast f\}_B\\
  &=&(f_{j(1)}\otimes g_j)\otimes_H(f_{j(-2)}b)\cdot_B b_j+(-1)^{|b||d|}(f_{m(1)}\otimes g_m)\otimes_H(f_{m(-2)}d)\cdot_B b_m,\\
  && \{d\ast (b\cdot_B f)\}_B=\{d\ast b\}_B\cdot_B f+(-1)^{|b||f|}\{d\ast f\}_B\cdot_B b\\
  &=&-(-1)^{|d||b|}(g_i\otimes f_{i(1)})\otimes_H b_i\cdot_B(f_{i(-2)}f)+(-1)^{|b||f|}(f_j\otimes g_{j(1)})\otimes_H b_j\cdot_B(g_{j(-2)}b),
\end{eqnarray*}
then we obtain
\begin{eqnarray*}
  &&\{(a\otimes b)\ast\{(c\otimes d)\ast(e\otimes f)\}\}=(-1)^{|d||e|}\{(a\otimes b)\ast\big((h_j\otimes l_j)\otimes_H(a_j\otimes (d\cdot _Bf))\big)\}\\
  &&+(-1)^{|d||e|}\{(a\otimes b)\ast\big((f_j\otimes g_j)\otimes_H((c\cdot _Ae)\otimes b_j)\big)\}\\
  &=&(-1)^{|d||e|+|b||c|+|b||e|)}
  \bigg((h_{jk}\otimes h_{j}l_{jk(1)}\otimes l_{j}l_{jk(2)})\otimes_H\big(a_{jk}\otimes(b\cdot_B(d\cdot_Bf))\big)\\
  &&+(f_{i}\otimes h_{j}g_{i(1)(1)}\otimes l_{j}g_{i(1)(2)})\otimes_H\big((a\cdot_A a_j)\otimes(b_i\cdot_B(g_{i(-2)}f))\big)\\
  &&+(-1)^{|f||d|}
  (f_{m}\otimes h_{j}g_{m(1)(1)}\otimes l_{j}g_{m(1)(2)})\otimes_H\big((a\cdot_A a_j)\otimes(b_m\cdot_B(g_{m(-2)}d))\big)\\
  &&+(h_{i}\otimes f_{j}l_{i(1)(1)}\otimes g_{j}l_{i(1)(2)})\otimes_H\big((a_i\cdot_A (l_{i(-2)}e))\otimes(b\cdot_Bb_j)\big)\\
  &&+(-1)^{|e||c|}
  (h_{m}\otimes f_{j}l_{m(1)(1)}\otimes g_{j}l_{m(1)(2)})\otimes_H\big((a_m\cdot_A (l_{m(-2)}c))\otimes(b\cdot_Bb_j)\big)\\
  &&+(f_{jk}\otimes f_{j}g_{jk(1)}\otimes g_{j}g_{jk(2)})\otimes_H\big((a\cdot_A(c\cdot_Ae))\otimes b_{jk}\big)\bigg),
\end{eqnarray*}

\begin{eqnarray*}
  &&\{\{(a\otimes b)\ast(c\otimes d)\}\ast(e\otimes f)\}=(-1)^{|b||c|}\{(h_i\otimes l_i)\otimes_H\big(a_i\otimes (b\cdot _Bd)\big)\ast(e\otimes f)\}\\
  &&+(-1)^{|b||c|}\{(f_i\otimes g_i)\otimes_H\big((a\cdot _Ac)\otimes b_i\big)\ast (e\otimes f)\}\\
  &=&(-1)^{|d||e|+|b||c|+|b||e|}\bigg(
  (h_{i}h_{ik(1)}\otimes l_{i}h_{ik(2)}\otimes l_{ik})\otimes_H\big(a_{ik}\otimes((b\cdot_Bd)\cdot_Bf)\big)\\
  &&+(h_{i}f_{j(1)(1)}\otimes l_{i}f_{j(1)(2)}\otimes g_{j})\otimes_H\big((a_i\cdot_A e)\otimes((f_{j(-2)}b)\cdot_Bb_j)\big)\\
  &&+(-1)^{|b||d|}
  (h_{i}f_{m(1)(1)}\otimes l_{i}f_{m(1)(2)}\otimes g_{m})\otimes_H\big((a_i\cdot_A e)\otimes((f_{m(-2)}d)\cdot_Bb_m)\big)\\
  &&+(f_{i}h_{j(1)(1)}\otimes g_{i}h_{j(1)(2)}\otimes l_{j})\otimes_H\big(((h_{j(-2)}a)\cdot_Aa_j)\otimes(b_i\cdot_Bf)\big)\\
  &&+(-1)^{|a||c|}
  (f_{i}h_{m(1)(1)}\otimes g_{i}h_{m(1)(2)}\otimes l_{m})\otimes_H\big(((h_{m(-2)}c)\cdot_Aa_m)\otimes(b_i\cdot_Bf)\big)\\
  &&+(f_{i}f_{ik(1)}\otimes g_{i}f_{ik(2)}\otimes g_{ik})\otimes_H\big(((a\cdot _Ac)\cdot_A e)\otimes b_{ik}\big)\bigg),
\end{eqnarray*}
and
\begin{eqnarray*}
  &&(-1)^{(|a|+|b|)(|c|+|d|)}\big((\sigma\otimes\mathrm{id})\otimes_H\mathrm{id}_{A\otimes B}\big)\{(c\otimes d)\ast\{(a\otimes b)\ast(e\otimes f)\}\}\\
  &=&(-1)^{|d||e|+|b||c|+|b||e|}\bigg((-1)^{|a||c|+|b||d|}
  (h_{m}l_{mk(1)}\otimes h_{mk}\otimes l_{m}l_{mk(2)})\otimes_H\big(a_{mk}\otimes(d\cdot_B(b\cdot_Bf))\big)\\
  &&-(-1)^{|a||c|}
  (h_{m}f_{i(1)(1)}\otimes g_{i}\otimes l_{m}f_{i(1)(2)})\otimes_H\big((c\cdot_A a_m)\otimes(b_i\cdot_B(f_{i(-2)}f))\big)\\
  &&+(-1)^{|a||c|+|b||d|+|b||f|}
  (h_{m}g_{j(1)(1)}\otimes f_{j}\otimes l_{m}g_{j(1)(2)})\otimes_H\big((c\cdot_A a_m)\otimes(b_j\cdot_B(g_{j(-2)}b))\big)\\
  &&-(-1)^{|b||d|}
  (f_{m}h_{i(1)(1)}\otimes l_{i}\otimes g_{m}h_{i(1)(2)})\otimes_H\big((a_i\cdot_A (h_{i(-2)}e))\otimes(d\cdot_Bb_m)\big)\\
  &&+(-1)^{|a||c|+|b||d|+|a||e|}
  (f_{m}l_{j(1)(1)}\otimes h_{j}\otimes g_{m}l_{j(1)(2)})\otimes_H\big((a_j\cdot_A (l_{j(-2)}a))\otimes(d\cdot_Bb_m)\big)\\
  &&+(-1)^{|a||c|+|b||d|}
  (f_{m}g_{mk(1)}\otimes f_{mk}\otimes g_{m}g_{mk(2)})\otimes_H\big((c\cdot_A(a\cdot_Ae))\otimes b_{mk}\big)\bigg).
\end{eqnarray*}
Hence, we obtain
\begin{eqnarray*}
  && \{(a\otimes b)\ast\{(c\otimes d)\ast(e\otimes f)\}\}-\{\{(a\otimes b)\ast(c\otimes d)\}\ast(e\otimes f)\}\\
  &&-(-1)^{(|a|+|b|)(|c|+|d|)}\big((\sigma\otimes\mathrm{id})\otimes_H\mathrm{id}_{A\otimes B}\big)\{(c\otimes d)\ast\{(a\otimes b)\ast(e\otimes f)\}\} \\
  &=& (-1)^{|d||e|+|b||c|+|b||e|}\big(A_1+B_1+\sum\limits_{i=1}^{3}(C_i+D_i)\big),
\end{eqnarray*}
where
\begin{eqnarray*}
  A_1&=&(h_{jk}\otimes h_{j}l_{jk(1)}\otimes l_{j}l_{jk(2)})\otimes_H\big(a_{jk}\otimes(b\cdot_B(d\cdot_Bf))\big) \\
  &&- (h_{i}h_{ik(1)}\otimes l_{i}h_{ik(2)}\otimes l_{ik})\otimes_H\big(a_{ik}\otimes((b\cdot_Bd)\cdot_Bf)\big)\\
  &&-(-1)^{|a||c|+|b||d|}
  (h_{m}l_{mk(1)}\otimes h_{mk}\otimes l_{m}l_{mk(2)})\otimes_H\big(a_{mk}\otimes(d\cdot_B(b\cdot_Bf))\big),\\
  B_1&=&(f_{jk}\otimes f_{j}g_{jk(1)}\otimes g_{j}g_{jk(2)})\otimes_H\big((a\cdot_A(c\cdot_Ae))\otimes b_{jk}\big)\\
  &&-(f_{i}f_{ik(1)}\otimes g_{i}f_{ik(2)}\otimes g_{ik})\otimes_H\big(((a\cdot _Ac)\cdot_A e)\otimes b_{ik}\big)\\
  &&-(-1)^{|a||c|+|b||d|}
  (f_{m}g_{mk(1)}\otimes f_{mk}\otimes g_{m}g_{mk(2)})\otimes_H\big((c\cdot_A(a\cdot_Ae))\otimes b_{mk}\big),\\
  C_1&=&(f_{i}\otimes h_{j}g_{i(1)(1)}\otimes l_{j}g_{i(1)(2)})\otimes_H\big((a\cdot_A a_j)\otimes(b_i\cdot_B(g_{i(-2)}f))\big)\\
  &&-(f_{i}h_{j(1)(1)}\otimes g_{i}h_{j(1)(2)}\otimes l_{j})\otimes_H\big(((h_{j(-2)}a)\cdot_Aa_j)\otimes(b_i\cdot_Bf)\big),\\
  C_2&=&(-1)^{|f||d|}\bigg(
  (f_{m}\otimes h_{j}g_{m(1)(1)}\otimes l_{j}g_{m(1)(2)})\otimes_H\big((a\cdot_A a_j)\otimes(b_m\cdot_B(g_{m(-2)}d))\big)\\
  &&-(f_{m}l_{j(1)(1)}\otimes h_{j}\otimes g_{m}l_{j(1)(2)})\otimes_H\big(((l_{j(-2)}a)\cdot_A a_j)\otimes(b_m\cdot_B d)\big)\bigg),\\
  C_3&=&(-1)^{|a||c|}\bigg(
  (h_{m}f_{i(1)(1)}\otimes g_{i}\otimes l_{m}f_{i(1)(2)})\otimes_H\big((c\cdot_A a_m)\otimes(b_i\cdot_B(f_{i(-2)}f))\big)\\
  &&-(f_{i}h_{m(1)(1)}\otimes g_{i}h_{m(1)(2)}\otimes l_{m})\otimes_H\big(((h_{m(-2)}c)\cdot_Aa_m)\otimes(b_i\cdot_Bf)\big)\bigg),\\
  D_1&=&(h_{i}\otimes f_{j}l_{i(1)(1)}\otimes g_{j}l_{i(1)(2)})\otimes_H\big((a_i\cdot_A (l_{i(-2)}e))\otimes(b\cdot_Bb_j)\big)\\
  &&-(h_{i}f_{j(1)(1)}\otimes l_{i}f_{j(1)(2)}\otimes g_{j})\otimes_H\big((a_i\cdot_A e)\otimes((f_{j(-2)}b)\cdot_Bb_j)\big),\\
  D_2&=&(-1)^{|e||c|}\bigg(
  (h_{m}\otimes f_{j}l_{m(1)(1)}\otimes g_{j}l_{m(1)(2)})\otimes_H\big((a_m\cdot_A (l_{m(-2)}c))\otimes(b\cdot_Bb_j)\big)\\
  &&-(h_{m}g_{j(1)(1)}\otimes f_{j}\otimes l_{m}g_{j(1)(2)})\otimes_H\big((a_m\cdot_A c)\otimes((g_{j(-2)}b)\cdot_B b_j)\big)\bigg),\\
  D_3&=&(-1)^{|b||d|}\bigg(
  (f_{m}h_{i(1)(1)}\otimes l_{i}\otimes g_{m}h_{i(1)(2)})\otimes_H\big((a_i\cdot_A (h_{i(-2)}e))\otimes(d\cdot_Bb_m)\big)\\
  &&-(h_{i}f_{m(1)(1)}\otimes l_{i}f_{m(1)(2)}\otimes g_{m})\otimes_H\big((a_i\cdot_A e)\otimes((f_{m(-2)}d)\cdot_Bb_m)\big)\bigg).
\end{eqnarray*}
Since $(A,\{\ast\}_A,d_A)$ and $(B,\{\ast\}_B,d_B)$ are two DGL pseudoalgebras, thus $A_1=B_1=0$.

By \eqref{eq:tensor-com1}, we have
\begin{eqnarray*}
  C_1&=&(f_i\otimes h_j\otimes l_j)(1\otimes\Delta)(1\otimes g_{i(1)})\otimes_H\big((a\cdot_A a_j)\otimes (b_i\cdot_B(g_{i(-2)}f))\big)\\
  &&-(f_i\otimes g_i\otimes l_j)(\Delta\otimes1)(h_{j(1)}\otimes1)\otimes_H\big(((h_{j(-2)}a)\cdot_Aa_j)\otimes(b_i\cdot_Bf)\big)=0,
\end{eqnarray*}
by \eqref{eq:tensor-com2}, 
\begin{eqnarray*}
  C_2&=&(-1)^{|f||d|}\bigg(
  (f_{m}\otimes h_{j}\otimes l_{j})(1\otimes\Delta)(1\otimes g_{m(1)})\otimes_H\big((a\cdot_A a_j)\otimes(b_m\cdot_B(g_{m(-2)}d))\big)\\
  &&-(f_{m}l_{j(1)(1)}\otimes h_{j}\otimes g_{m}l_{j(1)(2)})\otimes_H\big(((l_{j(-2)}a)\cdot_A a_j)\otimes(b_m\cdot_B d)\big)\bigg)\\
  &=&(-1)^{|f||d|}\bigg((\mathrm{id}\otimes\sigma)\otimes_H\mathrm{id}_{A\otimes B}\big)(f_{m}\otimes g_{m}\otimes h_{j})(\Delta\otimes1)(l_{j(1)}\otimes1)\otimes_H\big(((l_{j(-2)}a)\cdot_A a_j)\otimes(b_m\cdot_B d)\big)\\
  &&-(f_{m}l_{j(1)(1)}\otimes h_{j}\otimes g_{m}l_{j(1)(2)})\otimes_H\big(((l_{j(-2)}a)\cdot_A a_j)\otimes(b_m\cdot_B d)\big)\bigg)
  =0,
\end{eqnarray*}
by \eqref{eq:tensor-com3}, 
\begin{eqnarray*}
  C_3&=&-(-1)^{|a||c|}\bigg((f_{i}\otimes g_{i}\otimes l_{m})(\Delta\otimes1)(h_{m(1)}\otimes 1)\otimes_H
  \big(((h_{m(-2)}c)\cdot_Aa_m)\otimes(b_i\cdot_Bf)\big)\\
  &&-(h_{m}f_{i(1)(1)}\otimes g_{i}\otimes l_{m}f_{i(1)(2)})\otimes_H\big((c\cdot_A a_m)\otimes(b_i\cdot_B(f_{i(-2)}f))\big)\bigg)\\
  &=&-(-1)^{|a||c|}\bigg(\big((\sigma\otimes\mathrm{id})\otimes_H\mathrm{id}_{A\otimes B}\big)
  (g_{i}\otimes h_{m}\otimes l_{m})(1\otimes\Delta)(1\otimes f_{i(1)})\otimes_H\big((c\cdot_Aa_m)\otimes(b_i\cdot_B(f_{i(-2)}f))\big)\\
  &&-(h_{m}f_{i(1)(1)}\otimes g_{i}\otimes l_{m}f_{i(1)(2)})\otimes_H\big((c\cdot_A a_m)\otimes(b_i\cdot_B(f_{i(-2)}f))\big)\bigg)=0.
\end{eqnarray*}
Similarly, by \eqref{eq:tensor-com1} $\sim$ \eqref{eq:tensor-com3}, we have $D_1=D_2=D_3=0$.

Furthermore, we obtain
\begin{eqnarray*}
  &&\big((\mathrm{id}\otimes\mathrm{id})\otimes_H d\big)(\{(a\otimes b)\ast(a'\otimes b')\})\\
  &=&(-1)^{|b||a'|}\big((\mathrm{id}\otimes\mathrm{id})\otimes_H d\big)\big(\{a\ast a'\}_A\otimes(b\cdot_Bb')\big)+(-1)^{|b||a'|}\big((\mathrm{id}\otimes\mathrm{id})\otimes_H d\big)\big((a\cdot_Aa')\otimes\{b\ast b'\}_B\big)\\
  &=&(-1)^{|b||a'|}\big((\mathrm{id}\otimes\mathrm{id})\otimes_H d\big)\{a\ast a'\}_A\otimes(b\cdot_Bb')+(-1)^{|b||a'|+|a|+|a'|}\{a\ast a'\}_A\otimes d(b\cdot_B b')\\
  &&+(-1)^{|b||a'|}d(a\cdot_Aa')\otimes\{b\ast b'\}_B
  +(-1)^{|b||a'|+|a|+|a'|}(a\cdot_Aa')\otimes \big((\mathrm{id}\otimes\mathrm{id})\otimes_H d\big)\{b\ast b'\}_B\\
  &=&(-1)^{|b||a'|}\{d_A(a)\ast a'\}_A\otimes(b\cdot_Bb')+(-1)^{|b||a'|+|a|}\{a\ast d_A(a')\}_A\otimes(b\cdot_Bb')\\
  &&+(-1)^{|b||a'|+|a|+|a'|}\{a\ast a'\}_A\otimes\big(d_B(b)\cdot_B b'\big)+(-1)^{|b||a'|+|a|+|a'|+|b|}\{a\ast a'\}_A\otimes\big(b\cdot_B d_B(b')\big)\\
  &&+(-1)^{|b||a'|}\big(d_A(a)\cdot_Aa'\big)\otimes\{b\ast b'\}_B+(-1)^{|b||a'|+|a|}\big(a\cdot_Ad_A(a')\big)\otimes\{b\ast b'\}_B\\
  &&+(-1)^{|b||a'|+|a|+|a'|}(a\cdot_Aa')\otimes \{d_B(b)\ast b'\}_B+(-1)^{|b||a'|+|a|+|a'|+|b|}(a\cdot_Aa')\otimes \{b\ast d_B(b')\}_B\\
  &=&(-1)^{|b||a'|}\big(\{d_A(a)\ast a'\}_A\otimes(b\cdot_Bb')+(d_A(a)\cdot_Aa')\otimes\{b\ast b'\}_B\big)\\
  &&+(-1)^{|a|+|a'|+|a'||b|}\big(\{a\ast a'\}_A\otimes(d_B(b)\cdot_Bb')+(a\cdot_Aa')\otimes\{d_B(b)\ast b'\}_B\big)\\
  &&+(-1)^{|a|+|a'||b|}\big(\{a\ast d_A(a')\}_A\otimes(b\cdot_Bb')+(a\cdot_Ad_A(a'))\otimes\{b\ast b'\}_B\big)\\
  &&+(-1)^{|a|+|b|+|a'|+|a'||b|}\big(\{a\ast a'\}_A\otimes(b\cdot_Bd_B(b'))+(a\cdot_Aa')\otimes\{b\ast d_B(b')\}_B\big)\\
  &=&\{(d_A(a)\otimes b)\ast(a'\otimes b')\}+(-1)^{|a|}\{(a\otimes d_B(b))\ast(a'\otimes b')\}\\
  &&+(-1)^{|a|+|b|}\{(a\otimes b)\ast (d_A(a')\otimes b')\}+(-1)^{|a|+|b|+|a'|}\{(a\otimes b)\ast (a'\otimes d_B(b'))\}\\
  &=&\{d(a\otimes b)\ast(a'\otimes b')\}+(-1)^{|a|+|b|}\{(a\otimes b)\ast d(a'\otimes b')\},
\end{eqnarray*}
thus, $(A\otimes B,\{\ast\},d)$ is a DGL pseudoalgebra.

$H$-differential is true since
\begin{eqnarray*}
   h\big((a\otimes b)\star(a'\otimes b')\big)&=&(-1)^{|a'||b|}h\big((a\cdot_Aa')\otimes(b\cdot_Bb')\big) =(-1)^{|a'||b|}h_{(1)}(a\cdot_Aa')\otimes h_{(2)}(b\cdot_Bb') \\
   &=& (-1)^{|a'||b|}(h_{(1)}a)\cdot_A(h_{(2)}a')\otimes (h_{(3)}b)\cdot_B(h_{(4)}b') \\
   &=& \big((h_{(1)}a)\otimes(h_{(2)}b)\big)\star\big((h_{(3)}a')\otimes(h_{(4)}b')\big)\\
   &=& \big(h_{(1)}(a\otimes b)\big)\star\big(h_{(2)}(a'\otimes b')\big),
\end{eqnarray*}
graded left Leibniz rule \eqref{eq:graded-left-Leibniz-rule} follows from
\begin{eqnarray*}
   &&\{(a\otimes b)\ast\big((c\otimes d)\star(e\otimes f)\big)\}=(-1)^{|d||e|}\{(a\otimes b)\ast\big((c\cdot_A e)\otimes(d\cdot_B f)\big)\}\\
   &=&(-1)^{|d||e|+|b||c|+|b||e|}\bigg((h_i\otimes l_{i(1)})\otimes_H a_i\cdot_A(l_{i(-2)}e)\otimes b\cdot_B(d\cdot_Bf)\\
   &&+(-1)^{|e||c|}(h_m\otimes l_{m(1)})\otimes_H a_m\cdot_A(l_{m(-2)}c)\otimes b\cdot_B(d\cdot_Bf)\\
   &&+(f_i\otimes g_{i(1)})\otimes_Ha\cdot_A(c\cdot_Ae) \otimes b_i\cdot_B(g_{i(-2)}f)\\
   &&+(-1)^{|d||f|}(f_m\otimes g_{m(1)})\otimes_Ha\cdot_A(c\cdot_Ae)\otimes b_m\cdot_B(g_{m(-2)}d)\bigg),
\end{eqnarray*}   
and
\begin{eqnarray*}
   &&\{(a\otimes b)\ast(c\otimes d)\}\star(e\otimes f)+(-1)^{(|c|+|d|)(|e|+|f|)}\{(a\otimes b)\ast(e\otimes f)\}\star(c\otimes d)\\
   &=&(-1)^{|b||c|}\big((h_i\otimes l_i)\otimes_H a_i\otimes(b\cdot_Bd)+(f_i\otimes g_i)\otimes_H(a\cdot_Ac)\otimes b_i\big)\star(e\otimes f)\\
   &&+(-1)^{(|c|+|d|)(|e|+|f|)+|b||e|}\big((h_m\otimes l_m)\otimes_Ha_m\otimes(b\cdot_Bf)+(f_m\otimes g_m)\otimes_H(a\cdot_Ae)\otimes b_m\big)\star(c\otimes d)\\
   &=&(-1)^{|b||c|}\bigg((h_i\otimes l_{i(1)})\otimes_H\big(a_i\otimes(b\cdot_Bd)\star l_{i(-2)}(e\otimes f)\big)\\
   &&+(f_i\otimes g_{i(1)})\otimes_H\big((a\cdot_Ac)\otimes b_i\star g_{i(-2)}(e\otimes f)\big)\bigg)\\
   &&+(-1)^{(|c|+|d|)(|e|+|f|)+|b||e|}\bigg((h_m\otimes l_{m(1)})\otimes_H\big(a_m\otimes(b\cdot_Bf)\star l_{m(-2)}(c\otimes d)\big)\\
   &&+(f_m\otimes g_{m(1)})\otimes_H\big((a\cdot_Ae)\otimes b_m\star g_{m(-2)}(c\otimes d)\big)\bigg)\\
   &=&(-1)^{|b||c|}\bigg((h_i\otimes l_{i(1)})\otimes_H\big(a_i\otimes(b\cdot_Bd)\star (l_{i(-2)})_{(1)}e\otimes (l_{i(-2)})_{(2)}f\big)\\
   &&+(f_i\otimes g_{i(1)})\otimes_H\big((a\cdot_Ac)\otimes b_i\star (g_{i(-2)})_{(1)}e\otimes (g_{i(-2)})_{(2)}f\big)\bigg)\\
   &&+(-1)^{(|c|+|d|)(|e|+|f|)+|b||e|}\bigg((h_m\otimes l_{m(1)})\otimes_H\big(a_m\otimes(b\cdot_Bf)\star (l_{m(-2)})_{(1)}c\otimes (l_{m(-2)})_{(2)}d\big)\\
   &&+(f_m\otimes g_{m(1)})\otimes_H\big((a\cdot_Ae)\otimes b_m\star (g_{m(-2)})_{(1)}c\otimes (g_{m(-2)})_{(2)}d\big)\bigg)\\
   &=&(-1)^{|b||c|+(|b|+|d|)|e|}\bigg((h_i\otimes l_{i(1)})\otimes_H\big(a_i\cdot_A(l_{i(-2)(1)}e)\otimes (b\cdot_Bd)\cdot_Bl_{i(-2)(2)}f\big)\\
   &&+(f_i\otimes g_{i(1)})\otimes_H\big((a\cdot_Ac)\cdot_Ag_{i(-2)(1)}e\otimes b_i\cdot_B(g_{i(-2)(2)}f)\big)\\
   &&+(-1)^{|c||e|+|d||f|}(h_m\otimes l_{m(1)})\otimes_H\big(a_m\cdot_A(l_{m(-2)(1)}c)\otimes(b\cdot_Bf)\cdot_B l_{m(-2)(2)}d\big)\\
   &&+(-1)^{|c||e|+|d||f|}(f_m\otimes g_{m(1)})\otimes_H\big((a\cdot_Ae)\cdot_Ag_{m(-2)(1)}c\otimes b_m\cdot_B(g_{m(-2)(2)}d)\big)\bigg),
\end{eqnarray*}
by \eqref{eq:tensor-com4}, we have   
\begin{eqnarray*}
  && \{(a\otimes b)\ast\big((c\otimes d)\star(e\otimes f)\big)\} \\
  &=& \{(a\otimes b)\ast(c\otimes d)\}\star(e\otimes f)+(-1)^{(|c|+|d|)(|e|+|f|)}\{(a\otimes b)\ast(e\otimes f)\}\star(c\otimes d).
\end{eqnarray*}
Therefore, the proof is completed.
\end{proof}

%
%

\begin{pro}\label{pro:Poi-quotient-psalg}
Let $(A,\cdot,\{\ast\},d)$ be a DGP pseudoalgebra, and $B\lhd A$, for all homogeneous elements $a, b, c\in A$, if we define
\begin{eqnarray*}
  h(a+B) &:=& ha+B,\\
  (a+B) \cdot_{A/B} (b+B) &:=& a \cdot b+B,\\
  \big(\{a\ast b\}+B\big)\cdot_{A/B}(c+B) &:=& \{a\ast b\}\cdot c+B,\\
  \{(a+B) \ast (b+B)\}_{A/B} &:=& \{a \ast b\}+B, \\
  d_{A/B}(a+B) &:=& d(a)+B,\\
  \big((\mathrm{id}\otimes\mathrm{id})\otimes_H d_{A/B}\big)\big(\{a\ast b\}+B\big)&:=&\big((\mathrm{id}\otimes\mathrm{id})\otimes_H d\big)\{a \ast b\}+B,
\end{eqnarray*}
and we define 
\begin{equation}\label{eq:quotient}
  \big((f_1\otimes f_2\otimes\cdots\otimes f_n)\otimes_H a\big)+B=(f_1\otimes f_2\otimes\cdots\otimes f_n)\otimes_H(a+B),
\end{equation}
for $f_1, f_2,...,f_n\in H$,
then $(A/B,\cdot_{A/B},\{\ast\}_{A/B},d_{A/B})$ is a DGP pseudoalgebra.
\end{pro}

\begin{proof}
For all homogeneous elements $a, b, c\in A$, we have 
\begin{eqnarray*}
  d_{A/B}\big((a+B) \cdot_{A/B} (b+B)\big) &=& d_{A/B}(a\cdot b+B)=d(a\cdot b)+B\\
  &=&d(a)\cdot b+(-1)^{|a|}a\cdot d(b)+B\\
  &=& d_{A/B}(a+B)\cdot_{A/B}(b+B)+(-1)^{|a|}(a+B)\cdot_{A/B}d_{A/B}(b+B).
\end{eqnarray*}
Then $(A/B,\cdot_{A/B},d_{A/B})$ is a differential graded algebra.

Since $(A,\{\ast\})$ is a graded Lie $H$-pseudoalgebra and by \eqref{eq:quotient}, $(A/B,\{\ast\}_{A/B})$ is readily shown to be a graded Lie $H$-pseudoalgebra. Furthermore, we have
\begin{eqnarray*}
  &&\big((\mathrm{id}\otimes\mathrm{id})\otimes_H d_{A/B}\big)\{(a+B) \ast (b+B)\}_{A/B} = \big((\mathrm{id}\otimes\mathrm{id})\otimes_H d_{A/B}\big)\big(\{a\ast b\}+B\big)\\
  &=&\big((\mathrm{id}\otimes\mathrm{id})\otimes_H d\big)\{a \ast b\}+B 
  = \{d(a)\ast b\}+B+(-1)^{|a|+p}\{a\ast d(b)\}+B \\
  &=& \{d_{A/B}(a+B)\ast (b+B)\}_{A/B}+(-1)^{|a|+p}\{(a+B)\ast d_{A/B}(b+B)\}_{A/B}.
\end{eqnarray*}
Then $(A/B,\{\ast\}_{A/B},d_{A/B})$ is a DGL pseudoalgebra.

On the other hand, since $(A,\cdot,\{\ast\})$ is a graded Poisson $H$-pseudoalgebra, we obtain
\begin{eqnarray*}
  &&\{(a+B)\ast \big((b+B)\cdot_{A/B}(c+B)\big)\}_{A/B} = \{(a+B)\ast (b\cdot c+B)\}_{A/B}\\
  &=&\{a\ast (b\cdot c)\}+B=\{a\ast b\}\cdot c+B+(-1)^{(|c|+p)|b|}\{a\ast c\}\cdot b+B \\
  &=&  \{(a+B)\ast (b+B)\}_{A/B}\cdot_{A/B}(c+B)+(-1)^{(|c|+p)|b|}\{(a+B)\ast (c+B)\}_{A/B}\cdot_{A/B}(b+B),
\end{eqnarray*}
and
\begin{eqnarray*}
  &&h\big((a+B) \cdot_{A/B} (b+B)\big) = h(a\cdot b+B)=h(a\cdot b)+B=(h_{(1)}a)\cdot(h_{(2)}b)+B\\
  &=& (h_{(1)}a+B)\cdot_{A/B}(h_{(2)}b+B)=\big(h_{(1)}(a+B)\big)\cdot_{A/B}\big(h_{(2)}(b+B)\big).
\end{eqnarray*}
Hence, the conclusion holds.
\end{proof}

Recall that the cohomology theory of associative $H$-pseudoalgebras is introduced in \cite[Section 3]{Wu-2013-8}. Let $A$ be a left $H$-module in the pseudotensor category $\mathcal{M}^{\ast}(H)$, for $n\geq 0$, let
\begin{equation*}
  M^{n}(A) = Hom_{H^{\otimes (n+1)}}(A^{\boxtimes (n+1)},H^{\otimes (n+1)}\otimes_H A),
\end{equation*}
$M^{-1}(A)=A$ and $M^{n}(A)=0$ for $n\leq-2$.

Suppose $(A,\ast)$ is an associative $H$-pseudoalgebra and $d_n:~M^{n}(A)\rightarrow M^{n+1}(A)$, then $n$-th cohomology of $(A,\ast)$ is defined by 
\begin{equation*}
  \mathcal{H}^{n}(A)=\mathrm{ker}d_{n}/\mathrm{im}d_{n+1},
\end{equation*}
and we denote $\mathcal{H}(A)=\oplus_{n\in\mathbb{Z}}\mathcal{H}^{n}(A)$.

\begin{pro}\label{pro:coho-psalg}
Let $(A,\cdot_A,\{\ast\}_A,d_A)$ be a DGP pseudoalgebra, we define a DGP pseudoalgebra structure $\big(\mathcal{H}(A),\star_{\mathcal{H}(A)},\{\ast\}_{\mathcal{H}(A)},d_{\mathcal{H}(A)}\big)$ on the cohomology $H$-pseudoalgebra $\mathcal{H}(A)$ of $(A,\cdot_A,\{\ast\}_A,d_A)$ by
\begin{eqnarray*}
  h(a+\mathrm{im}d_A) &:=& ha+\mathrm{im}d_A, \\
  (a+\mathrm{im}d_A) \star_{\mathcal{H}(A)} (b+\mathrm{im}d_A) &:=& a\cdot_A b+\mathrm{im}d_A, \\
  \big(\{a\ast b\}_{A}+\mathrm{im}d_{A}\big)\star_{\mathcal{H}(A)}(c+\mathrm{im}d_A) &:=& \{a\ast b\}_A\cdot_A c+\mathrm{im}d_A,\\
   \{(a+\mathrm{im}d_A)\ast (b+\mathrm{im}d_A)\}_{\mathcal{H}(A)}&:=&\{a\ast b\}_A+\mathrm{im}d_A,\\
  d_{\mathcal{H}(A)}(a+\mathrm{im}d_A) &:=& d_A(a)+\mathrm{im}d_A,\\
  \big((\mathrm{id}\otimes\mathrm{id})\otimes_Hd_{\mathcal{H}(A)}\big)\big(\{a\ast b\}_{\mathcal{H}(A)}+\mathrm{im}d_{A}\big)&:=&\big((\mathrm{id}\otimes\mathrm{id})\otimes_H d_A\big)\{a\ast b\}_A+\mathrm{im}d_A,
\end{eqnarray*}
for all homogeneous elements $a+\mathrm{im}d_A, b+\mathrm{im}d_A, c+\mathrm{im}d_A\in \mathcal{H}(A)$. Moreover, if $(B,\cdot_B,\{\ast\}_B,d_B)$ is another DGP pseudoalgebra such that $A\cong B$ as DGP pseudoalgebras, then $\mathcal{H}(A)\cong \mathcal{H}(B)$ as DGP pseudoalgebras.
\end{pro}
\begin{proof}
The DGP pseudoalgebra structure on the cohomology $H$-pseudoalgebra $\mathcal{H}(A)$ follows immediately from Proposition \ref{pro:Poi-quotient-psalg}.

Let $f:~A\rightarrow B$ be a DGP pseudoalgebra isomorphism. Define an $H$-linear map $\widetilde{f}:~\mathcal{H}(A)\rightarrow \mathcal{H}(B)$ by
\begin{equation*}
  \widetilde{f}(a+\mathrm{im}d_A)=f(a)+\mathrm{im}d_B, \quad\forall a+\mathrm{im}d_A \in \mathcal{H}(A).
\end{equation*}
For all homogeneous elements $a+\mathrm{im}d_A, a'+\mathrm{im}d_A\in \mathcal{H}(A)$, if $a+\mathrm{im}d_A=a'+\mathrm{im}d_A$, then $a-a'\in \mathrm{im}d_A$. By the definition of a DGP pseudoalgebra isomorphism, we have $f(a-a')\in \mathrm{im}d_B$ and $f(a)+\mathrm{im}d_B=f(a')+\mathrm{im}d_B$. Then $\widetilde{f}$ is well defined. 
By the DGP pseudoalgebra structure $\big(\mathcal{H}(A),\star_{\mathcal{H}(A)},\{\ast\}_{\mathcal{H}(A)},d_{\mathcal{H}(A)}\big)$ defined above and \eqref{eq:quotient}, we have
\begin{eqnarray*}
  \widetilde{f}\big((a+\mathrm{im}d_A)\star_{\mathcal{H}(A)}(a'+\mathrm{im}d_A)\big)
   &=&\widetilde{f}(a\cdot_A a'+\mathrm{im}d_A)=f(a\cdot_A a')+\mathrm{im}d_B\\
  &=&f(a)\cdot_Bf(a')+\mathrm{im}d_B=(f(a)+\mathrm{im}d_B)\star_{\mathcal{H}(B)}(f(a')+\mathrm{im}d_B) \\
  &=&\widetilde{f}(a+\mathrm{im}d_A)\star_{\mathcal{H}(B)}\widetilde{f}(a'+\mathrm{im}d_A),
\end{eqnarray*}
\begin{eqnarray*}
  \big((\mathrm{id}\otimes\mathrm{id})\otimes_H \widetilde{f}\big)\{(a+\mathrm{im}d_A)\ast(a'+\mathrm{im}d_A)\}_{\mathcal{H}(A)} &=& \big((\mathrm{id}\otimes\mathrm{id})\otimes_H \widetilde{f}\big)\big(\{a\ast a'\}_A+\mathrm{im}d_A \big)\\
  &=&\big((\mathrm{id}\otimes\mathrm{id})\otimes_H f\big)\{a\ast a'\}_A+\mathrm{im}d_B\\
  &=&\{(f(a))\ast (f(a'))\}_B+\mathrm{im}d_B\\
  &=&\{(f(a)+\mathrm{im}d_B)\ast(f(a')+\mathrm{im}d_B)\}_{\mathcal{H}(B)}\\
  &=&\{\widetilde{f}(a+\mathrm{im}d_A)\ast\widetilde{f}(a'+\mathrm{im}d_A)\}_{\mathcal{H}(B)},
\end{eqnarray*}
and
\begin{eqnarray*}
  && (\widetilde{f}\circ d_{\mathcal{H}(A)})(a+\mathrm{im}d_A)=\widetilde{f}(d_A(a)+\mathrm{im}d_A)=f(d_A(a))+\mathrm{im}d_B \\
  &&= d_B(f(a))+\mathrm{im}d_B=d_{\mathcal{H}(B)}(f(a)+\mathrm{im}d_B)=(d_{\mathcal{H}(B)}\circ\widetilde{f})(a+\mathrm{im}d_A).
\end{eqnarray*}
Then $\widetilde{f}$ is a DGP pseudoalgebra homomorphism. Since $f$ is a DGP pseudoalgebra isomorphism, we deduce that $\widetilde{f}$ is bijective. Therefore, $\mathcal{H}(A)\cong \mathcal{H}(B)$ as DGP pseudoalgebras.
Hence the conclusion holds.
\end{proof}

\subsection{Universal enveloping $H$-pseudoalgebras of DGP pseudoalgebras}\label{subsec:Uepa-DPpa} 

Let $(A,\cdot,\{\cdot,\cdot\},d)$ be a differential graded Poisson algebra with an identity element 1. Denote $M_A=\{M_a: a\in A\}$ and $H_A=\{H_a: a\in A\}$ as two copies of the graded vector space $A$ endowed with two linear isomorphisms $M : A \rightarrow M_A ~(a \mapsto M_a)$ and $H : A \rightarrow H_A ~(a \mapsto H_a)$. From \cite[Section 3]{Lv-2016-4}, the universal enveloping algebra $A^{ue}$ of $(A,\cdot,\{\cdot,\cdot\},d)$ can be explicitly generated by $M_A$ and $H_A$ subject to the following relations:
\begin{eqnarray*}
  M_{a\cdot b} &=& M_{a}M_{b}; \\
  H_{\{a,b\}} &=& H_{a}H_{b}-(-1)^{(|a|+p)(|b|+p)}H_{b}H_{a}; \\
  M_{\{a,b\}}&=& H_{a}M_{b}-(-1)^{(|a|+p)|b|}M_{b}H_{a}; \\
  H_{a\cdot b} &=& M_{a}H_{b}+(-)^{|a||b|}M_{b}H_{a}; \\
  M_1&=& 1,
\end{eqnarray*}
for all homogeneous elements $a, b\in A$.

The similar construction can be applied to DGP pseudoalgebras, which is defined by the similar relations as those for differential graded Poisson algebras.

\begin{defi}\label{def:Uni-enve-pseudoalg}
Let $(A,\cdot,\{\ast\},d)$ be a DGP pseudoalgebra, and
\begin{equation*}
  M_A=\{M_a: a \in A\},\quad H_A=\{H_a: a \in A\}
\end{equation*}
are two copies of the graded left $H$-module $A$ endowed with two $H$-linear isomorphisms $M:A\rightarrow M_A$ and $H:A\rightarrow H_A$ sending $a$ to $M_a$ and $H_a$ respectively. The {\bf universal enveloping $H$-pseudoalgebra} $(A^{ue},\ast)$ of $(A,\cdot,\{\ast\},d)$ is defined to be the quotient $H$-pseudoalgebra of the free $H$-pseudoalgebra generated by $M_A$ and $H_A$, subject to the following relations:
\begin{eqnarray}
  \label{eq:rela-M-pro}(1\otimes1)\otimes_H M_{a\cdot b} &=& M_{a}\ast M_{b}, \\
  \label{eq:rela-H-bra}\big((\mathrm{id}\otimes\mathrm{id})\otimes_HH\big)_{\{a\ast b\}} &=& H_{a}\ast H_{b}-(-1)^{(|a|+p)(|b|+p)}(\sigma\otimes_H\mathrm{id})H_{b}\ast H_{a}, \\
  \label{eq:rela-M-bra}\big((\mathrm{id}\otimes\mathrm{id})\otimes_HM\big)_{\{a\ast b\}}&=& H_{a}\ast M_{b}-(-1)^{(|a|+p)|b|}(\sigma\otimes_H\mathrm{id})M_{b}\ast H_{a}, \\
  \label{eq:rela-H-pro}(1\otimes1)\otimes_HH_{a\cdot b} &=& M_{a}\ast H_{b}+(-1)^{|a||b|}M_{b}\ast H_{a}, \\
  \label{eq:rela-M-unit}(1\otimes1)\otimes_H M_1&=& 1\ast 1,
\end{eqnarray}
for all homogeneous elements $a, b \in A$, where $\big((\mathrm{id}\otimes\mathrm{id})\otimes_HH\big)_{\{a\ast b\}}=\sum\limits_{i}(f_i\otimes g_i)\otimes_H H_{e_i}$ for $\{a\ast b\}=\sum\limits_{i}(f_i\otimes g_i)\otimes_H e_i$.
\end{defi}

By the above construction of $(A^{ue},\ast)$, we can obtain that $(A^{ue},\ast)$ is an associative $H$-pseudoalgebra with identity $M_1$. Next, we construct a new DG pseudoalgebra structure on $(A^{ue},\ast)$ from the DGP pseudoalgebra.
\begin{pro}\label{pro:uni-enve-pseudoalg}
Let $(A,\cdot,\{\ast\},d)$ be a DGP pseudoalgebra and $(A^{ue},\ast)$ be the universal enveloping $H$-pseudoalgebra of $(A,\cdot,\{\ast\},d)$. If $\mathfrak{D}:~A^{ue}\rightarrow A^{ue}$ is an $H$-linear map such that for all homogeneous elements $a\in A$,
\begin{eqnarray*}
  && |M_a|=|a|,\quad|H_a|=|a|+p,\quad\mathfrak{D}(M_a)=M_{d(a)},\quad\mathfrak{D}(H_a)=H_{d(a)}, \\
  && \big((\mathrm{id}\otimes\mathrm{id})\otimes_H\mathfrak{D}\big)\big((\mathrm{id}\otimes\mathrm{id})\otimes_HM\big)_{\{a\ast b\}} = \big((\mathrm{id}\otimes\mathrm{id})\otimes_HM\big)_{\big((\mathrm{id}\otimes\mathrm{id})\otimes_Hd\big)\{a\ast b\}},\\
  && \big((\mathrm{id}\otimes\mathrm{id})\otimes_H\mathfrak{D}\big)\big((\mathrm{id}\otimes\mathrm{id})\otimes_HH\big)_{\{a\ast b\}} = \big((\mathrm{id}\otimes\mathrm{id})\otimes_HH\big)_{\big((\mathrm{id}\otimes\mathrm{id})\otimes_Hd\big)\{a\ast b\}},
\end{eqnarray*}
\noindent then $\mathfrak{D}$ is a differential graded operator.
\end{pro}
\begin{proof}
It suffices to prove that $\mathfrak{D}$ preserves all the relations \eqref{eq:rela-M-pro}$\sim$\eqref{eq:rela-M-unit}. It is obvious that $\big((\mathrm{id}\otimes\mathrm{id})\otimes_H\mathfrak{D}\big)\big((1\otimes1)\otimes_H M_1\big)
=(1\otimes1)\otimes_H\mathfrak{D}(M_1)=\big((\mathrm{id}\otimes\mathrm{id})\otimes_H\mathfrak{D}\big)(1\ast 1)$.
For all homogeneous elements $a,b\in A$, by \eqref{eq:rela-M-pro} we have
\begin{eqnarray*}
  &&\big((\mathrm{id}\otimes\mathrm{id})\otimes_H\mathfrak{D}\big)\big((1\otimes1)\otimes_H M_{a\cdot b}\big) = (1\otimes1)\otimes_HM_{d(a\cdot b)}\\
  &=&(1\otimes1)\otimes_H(M_{d(a)\cdot b}+(-1)^{|a|}M_{a\cdot d(b)}) 
  = M_{d(a)}\ast M_b+(-1)^{|M_a|}M_a\ast M_{d(b)} \\
  &=& \mathfrak{D}(M_a)\ast M_b+(-1)^{|M_a|}M_a\ast\mathfrak{D}(M_b)
  = \big((\mathrm{id}\otimes\mathrm{id})\otimes_H\mathfrak{D}\big)(M_a\ast M_b).
\end{eqnarray*}
By \eqref{eq:rela-H-bra} we obtain
\begin{eqnarray*}
  &&\big((\mathrm{id}\otimes\mathrm{id})\otimes_H\mathfrak{D}\big)\big((\mathrm{id}\otimes\mathrm{id})\otimes_HH\big)_{\{a\ast b\}} = \big((\mathrm{id}\otimes\mathrm{id})\otimes_HH\big)_{\big((\mathrm{id}\otimes\mathrm{id})\otimes_Hd\big)\{a\ast b\}}\\
  &=&\big((\mathrm{id}\otimes\mathrm{id})\otimes_HH\big)_{\{d(a)\ast b\}}+(-1)^{|a|+p}\big((\mathrm{id}\otimes\mathrm{id})\otimes_HH\big)_{\{a\ast d(b)\}} \\
  &\overset{\eqref{eq:rela-H-bra}}{=}& H_{d(a)}\ast H_b-(-1)^{(|d(a)|+p)(|b|+p)}(\sigma\otimes_H\mathrm{id})H_b\ast H_{d(a)}\\
  &+&(-1)^{|a|+p}\big(H_a\ast H_{d(b)}-(-1)^{(|a|+p)(|d(b)|+p)}(\sigma\otimes_H\mathrm{id})H_{d(b)}\ast H_a\big)\\
  &=& \mathfrak{D}(H_a)\ast H_b+(-1)^{|H_a|}H_a\ast\mathfrak{D}(H_b)\\
  &-&(-1)^{(|a|+p)(|b|+p)}(\sigma\otimes_H\mathrm{id})\big(\mathfrak{D}(H_b)\ast H_a+(-1)^{|H_b|}H_b\ast\mathfrak{D}(H_a)\big) \\
  &=& \big((\mathrm{id}\otimes\mathrm{id})\otimes_H\mathfrak{D}\big)(H_{a}\ast H_{b}-(-1)^{(|a|+p)(|b|+p)}(\sigma\otimes_H\mathrm{id})H_{b}\ast H_{a}),
\end{eqnarray*}
where the above notation $\overset{\eqref{eq:rela-H-bra}}{=}$ indicates that the equation \eqref{eq:rela-H-bra} is used to establish the equality. Furthermore, by \eqref{eq:rela-M-bra} we have
\begin{eqnarray*}
  &&\big((\mathrm{id}\otimes\mathrm{id})\otimes_H\mathfrak{D}\big)\big((\mathrm{id}\otimes\mathrm{id})\otimes_HM\big)_{\{a\ast b\}} = \big((\mathrm{id}\otimes\mathrm{id})\otimes_HM\big)_{\big((\mathrm{id}\otimes\mathrm{id})\otimes_Hd\big)\{a\ast b\}}\\
  &=&\big((\mathrm{id}\otimes\mathrm{id})\otimes_HM\big)_{\{d(a)\ast b\}}+(-1)^{|a|+p}\big((\mathrm{id}\otimes\mathrm{id})\otimes_HM\big)_{\{a\ast d(b)\}} \\
  &\overset{\eqref{eq:rela-M-bra}}{=}& \big(H_{d(a)}\ast M_b-(-1)^{(|d(a)|+p)|b|}(\sigma\otimes_H\mathrm{id})M_b\ast H_{d(a)}\big)\\
  &+&(-1)^{|a|+p}\big(H_a\ast M_{d(b)}-(-1)^{(|a|+p)|d(b)|}(\sigma\otimes_H\mathrm{id})M_{d(b)}\ast H_a\big) \\
  &=& \big(\mathfrak{D}(H_a)\ast M_b+(-1)^{|H_a|}H_a\ast \mathfrak{D}(M_b)\big)\\
  &-&(-1)^{(|a|+p)|b|}(\sigma\otimes_H\mathrm{id})\big(\mathfrak{D}(M_b)\ast H_a+(-1)^{|M_b|}M_b\ast \mathfrak{D}(H_a)\big) \\
  &=& \big((\mathrm{id}\otimes\mathrm{id})\otimes_H\mathfrak{D}\big)(H_{a}\ast M_{b}-(-1)^{(|a|+p)|b|}(\sigma\otimes_H\mathrm{id})M_{b}\ast H_{a}).
\end{eqnarray*}

By \eqref{eq:rela-H-pro} we obtain
\begin{eqnarray*}
  &&\big((\mathrm{id}\otimes\mathrm{id})\otimes_H\mathfrak{D}\big)\big((1\otimes1)\otimes_HH_{a\cdot b}\big) = (1\otimes1)\otimes_HH_{d(a\cdot b)}
  =(1\otimes1)\otimes_H\big(H_{d(a)\cdot b}+(-1)^{|a|}H_{a\cdot d(b)}\big) \\
  &\overset{\eqref{eq:rela-H-pro}}{=}& \big(M_{d(a)}\ast H_b+(-1)^{|d(a)||b|}M_b\ast H_{d(a)}\big)
  +(-1)^{|a|}\big(M_a\ast H_{d(b)}+(-1)^{|d(b)||a|}M_{d(b)}\ast H_a\big) \\
  &=& \big(\mathfrak{D}(M_a)\ast H_b+(-1)^{|M_a|}M_a\ast\mathfrak{D}(H_b)\big)
  +(-1)^{|a||b|}\big(\mathfrak{D}(M_b)\ast H_a+(-1)^{|M_b|}M_b\ast\mathfrak{D}(H_a)\big)\\
  &=& \big((\mathrm{id}\otimes\mathrm{id})\otimes_H\mathfrak{D}\big)(M_{a}\ast H_{b}+(-1)^{|a||b|}M_{b}\ast H_{a}).
\end{eqnarray*}
Hence, the conclusion holds.

\end{proof}

\begin{defi}\label{defi:P-triple}
Let $(A,\cdot_A,\{\ast\}_A,d_A)$ be a DGP pseudoalgebra and $(B,\ast_B,d_B)$ be a DG pseudoalgebra,
for all homogeneous elements $a,b \in A$, if
\begin{itemize}
  \item [($\mathcal{P}$1)]~$f:~A\rightarrow B$ is an $H$-linear map, satisfying $(1\otimes1)\otimes_Hf(a\cdot_Ab)=f(a)\ast_Bf(b)$;
  \item [($\mathcal{P}$2)]~$g:(A,\{\ast\}_A,d_A)\rightarrow (B,\{\ast\}_B,d_B)$ is a DGL pseudoalgebra homomorphism, i.e., $\big((\mathrm{id}\otimes\mathrm{id})\otimes_Hg\big)\{a\ast b\}_A=\{g(a)\ast g(b)\}_B$, where $\{\ast\}_B$ is defined as the commutator: $\{x\ast y\}_B=x\ast_B y-(-1)^{(|x|+p)(|y|+p)}(\sigma\otimes_H \mathrm{id})(y\ast_B x)$, for all homogeneous elements $x, y \in B$;
  \item [($\mathcal{P}$3)]~$\big((\mathrm{id}\otimes\mathrm{id})\otimes_Hf\big)\{a\ast b\}_A=g(a)\ast_Bf(b)-(-1)^{(|a|+p)|b|}(\sigma\otimes_H \mathrm{id})\big(f(b)\ast_Bg(a)\big)$;
  \item [($\mathcal{P}$4)]~$(1\otimes1)\otimes_Hg(a\cdot_Ab)=f(a)\ast_Bg(b)+(-1)^{|a||b|}f(b)\ast_Bg(a)$,
\end{itemize}
then the triple $\big((B,\ast_B,d_B),f,g\big)$ is called a {\bf $\mathcal{P}$-triple} of $(A,\cdot_A,\{\ast\}_A,d_A)$.

\end{defi}

\begin{rmk}\label{rmk:ex-triple}
Note that in Definition \ref{def:Uni-enve-pseudoalg}, $M_A$ and $H_A$ induce two $H$-linear maps $M$ and $H$ from $(A,\cdot,\{\ast\},d)$ to $(A^{ue},\ast)$, and the above properties $(\mathcal{P}1)\sim (\mathcal{P}4)$ are respectively equivalent to \eqref{eq:rela-M-pro} $\sim$ \eqref{eq:rela-H-pro}. Hence, the triple $\big((A^{ue},\ast,\mathfrak{D}),M,H\big)$ is a $\mathcal{P}$-triple of $(A,\cdot_A,\{\ast\}_A,d_A)$.
\end{rmk}

\begin{thm}\label{thm:bi-commute}
Let $\big((B,\ast_B,d_B),f,g\big)$ be a $\mathcal{P}$-triple of a DGP pseudoalgebra $(A,\cdot_A,\{\ast\}_A,d_A)$. Then there exists a unique DG pseudoalgebra homomorphism $\phi:~(A^{ue},\ast,\mathfrak{D})\rightarrow (B,\ast_B,d_B)$ satisfying $f=\phi\circ M$ and $g=\phi\circ H$, i.e., the following diagram bi-commutes:
\[
 \xymatrix{
 (A,\cdot_A,\{\ast\}_A,d_A) \ar[d]_{H} \ar[d]^{M} \ar[r]_{\quad g} \ar[r]^{\quad f} & (B,\ast_B,d_B)     \\
 (A^{ue},\ast,\mathfrak{D}) \ar@{.>}[ur]|-{\exists!\phi}                               }
\]

\end{thm}

\begin{proof}
By replacing the differential graded Poisson algebra in \cite[Proposition 4.3]{Lv-2016-4} with a DGP pseudoalgebra and following the similar proof strategy, we next prove that the DG pseudoalgebra $(A^{ue},\ast,\mathfrak{D})$ constructed in Proposition \ref{pro:uni-enve-pseudoalg} is universal with respect to the DGP pseudoalgebra.
Assume that the homomorphism $\phi$ exists. For all homogeneous elements $a \in A$, using the bi-commutativity relations $f = \phi \circ M$ and $g = \phi \circ H$, define an $H$-linear map $\phi:~(A^{ue},\ast,\mathfrak{D})\rightarrow (B,\ast_B,d_B)$ by 
\begin{equation*}
\phi(M_a) = f(a), \quad \phi(H_a) = g(a),
\end{equation*}
where $|g(a)|=|a|+p$.

Firstly, we prove that $\phi$ preserves the relations \eqref{eq:rela-M-pro} $\sim$ \eqref{eq:rela-M-unit} in Definition \ref{def:Uni-enve-pseudoalg}. By ($\mathcal{P}1$), for all homogeneous elements $a,b\in A$, we have
\begin{eqnarray*}
  &&\big((\mathrm{id}\otimes\mathrm{id})\otimes_H\phi\big)\big((1\otimes1)\otimes_HM_{a\cdot b}\big) = (1\otimes1)\otimes_H\phi(M_{a\cdot b})=(1\otimes1)\otimes_Hf(a\cdot_Ab) \\
  &=& f(a)\ast_Bf(b)=\phi(M_a)\ast_B \phi(M_b)=\big((\mathrm{id}\otimes\mathrm{id})\otimes_H\phi\big)(M_a\ast M_b),
\end{eqnarray*}
and $\big((\mathrm{id}\otimes\mathrm{id})\otimes_H\phi\big)\big((1\otimes1)\otimes_HM_{1}\big)=(1\otimes1)\otimes_Hf(1)
=(1\otimes1)\otimes_H1=\big((\mathrm{id}\otimes\mathrm{id})\otimes_H\phi\big)(1\ast1)$, thus $\phi$ preserves \eqref{eq:rela-M-pro} and \eqref{eq:rela-M-unit}. For all homogeneous elements $a,b\in A$, ($\mathcal{P}2$) implies that
\begin{eqnarray*}
  &&\big((\mathrm{id}\otimes\mathrm{id})\otimes_H\phi\big)\big((\mathrm{id}\otimes\mathrm{id})\otimes_HH\big)_{\{a\ast b\}_A} = \big((\mathrm{id}\otimes\mathrm{id})\otimes_H g\big)\{a\ast b\}_A=\{g(a)\ast g(b)\}_B \\
  &=& g(a)\ast_B g(b)-(-1)^{(|a|+p)(|b|+p)}(\sigma\otimes_H \mathrm{id})g(b)\ast_B g(a)\\
  &=& \phi(H_{a})\ast_B \phi(H_{b})-(-1)^{(|a|+p)(|b|+p)}(\sigma\otimes_H\mathrm{id})\phi(H_{b})\ast_B \phi(H_{a})\\
  &=& \big((\mathrm{id}\otimes\mathrm{id})\otimes_H\phi\big)(H_{a}\ast H_{b})-(-1)^{(|a|+p)(|b|+p)}\big((\mathrm{id}\otimes\mathrm{id})\otimes_H\phi\big)\big((\sigma\otimes_H\mathrm{id})H_{b}\ast H_{a}\big),
\end{eqnarray*}
hence $\phi$ preserves \eqref{eq:rela-H-bra}. 

Since
\begin{eqnarray*}
  &&\big((\mathrm{id}\otimes\mathrm{id})\otimes_H\phi\big)\big((\mathrm{id}\otimes\mathrm{id})\otimes_HM\big)_{\{a\ast b\}_A} = \big((\mathrm{id}\otimes\mathrm{id})\otimes_H f\big)\{a\ast b\}_A\\
  &=& g(a)\ast_B f(b)-(-1)^{(|a|+p)|b|}(\sigma\otimes_H \mathrm{id})f(b)\ast_B g(a)\\
  &=& \phi(H_{a})\ast_B \phi(M_{b})-(-1)^{(|a|+p)|b|}(\sigma\otimes_H\mathrm{id})\phi(M_{b})\ast_B \phi(H_{a})\\
  &=& \big((\mathrm{id}\otimes\mathrm{id})\otimes_H\phi\big)(H_{a}\ast M_{b})-(-1)^{(|a|+p)|b|}\big((\mathrm{id}\otimes\mathrm{id})\otimes_H\phi\big)\big((\sigma\otimes_H\mathrm{id})M_{b}\ast H_{a}\big),
\end{eqnarray*}
and
\begin{eqnarray*}
  &&\big((\mathrm{id}\otimes\mathrm{id})\otimes_H\phi\big)\big((1\otimes1)\otimes_HH_{a\cdot b}\big) = (1\otimes1)\otimes_H\phi(H_{a\cdot b})=(1\otimes1)\otimes_Hg(a\cdot_Ab) \\
  &=& f(a)\ast_Bg(b)+(-1)^{|a||b|}f(b)\ast_Bg(a)=\phi(M_a)\ast_B \phi(H_b)+(-1)^{|a||b|}\phi(M_b)\ast_B \phi(H_a)\\
  &=&\big((\mathrm{id}\otimes\mathrm{id})\otimes_H\phi\big)(M_a\ast H_b+(-1)^{|a||b|}M_b\ast H_a),
\end{eqnarray*}
thus, $\phi$ keeps \eqref{eq:rela-M-bra} and \eqref{eq:rela-H-pro} because of ($\mathcal{P}3$) and ($\mathcal{P}4$).

Secondly, we show that $\phi$ commutes with $\mathfrak{D}$ in Proposition \ref{pro:uni-enve-pseudoalg}. For all homogeneous elements $a\in A$, we have
\begin{align*}
(\phi\circ \mathfrak{D})(M_a) &= \phi(M_{d_A(a)}) = f(d_A(a)) = d_B(f(a)) = (d_B\circ\phi)(M_a), \\
(\phi\circ \mathfrak{D})(H_a) &= \phi(H_{d_A(a)}) = g(d_A(a)) = d_B(g(a)) = (d_B\circ\phi)(H_a),
\end{align*}
and it is clear that $\phi$ is unique. Hence the conclusion holds.
\end{proof}

\end{document}